\documentclass[12pt]{article}
\usepackage{amsmath, amsthm, amssymb,amscd, verbatim,hyperref,mathrsfs,authblk,cleveref}
\usepackage[all]{xy}
\theoremstyle{plain}
\newtheorem{thm}{Theorem}[section]

\newtheorem{lem}[thm]{Lemma}
\newtheorem{cor}[thm]{Corollary}

\newtheorem{prop}[thm]{Proposition}

\theoremstyle{definition}
\newtheorem{defn}[thm]{Definition}

\theoremstyle{remark}
\newtheorem{rem}[thm]{Remark}

\newcommand{\nc}{\newcommand} 
\nc{\hb}{\mathbb} 
\nc{\M}{\mathcal} 
\nc{\mf}{\mathfrak}
\nc{\mbf}{\mathbf}
\nc{\DMO}{\DeclareMathOperator}

\newbox\noforkbox \newdimen\forklinewidth
\setbox0\hbox{$\textstyle\smile$}
\setbox1\hbox to \wd0{\hfil\vrule width \forklinewidth depth-2pt
 height 10pt \hfil}
\setbox\noforkbox\hbox{\lower 2pt\box1\lower 2pt\box0\relax}
\def\anchor{\mathop{\copy\noforkbox}\limits}
 
\setbox0\hbox{$\textstyle\smile$}
\setbox1\hbox to \wd0{\hfil{\sl /\/}\hfil}
\setbox2\hbox to \wd0{\hfil\vrule height 10pt depth -2pt width
               \forklinewidth\hfil}
\newbox\doesforkbox
\setbox\doesforkbox\hbox{\box1 \lower 2pt\box2\lower2pt\box0\relax}
\def\nanchor{\mathop{\copy\doesforkbox}\limits}
 
\nc{\cA}{{\M A}} \nc{\cB}{{\M B}} \nc{\cC}{{\M C}} \nc{\cD}{{\M D}}
\nc{\cE}{{\M E}} \nc{\cF}{{\M F}} \nc{\cG}{{\M G}} \nc{\cH}{{\M H}}
\nc{\cI}{{\M I}} \nc{\cJ}{{\M J}} \nc{\cK}{{\M K}} \nc{\cL}{{\M L}}
\nc{\cM}{{\M M}} \nc{\cN}{{\M N}} \nc{\cO}{{\M O}} \nc{\cP}{{\M P}}
\nc{\cQ}{{\M Q}} \nc{\cR}{{\M R}} \nc{\cS}{{\M S}} \nc{\cT}{{\M T}}
\nc{\cU}{{\M U}} \nc{\cV}{{\M V}} \nc{\cW}{{\M W}} \nc{\cX}{{\M X}}
\nc{\cY}{{\M Y}} \nc{\cZ}{{\M Z}}
\nc{\Aa}{{\hb A}} \nc{\Cc}{{\hb C}} \nc{\Gg}{{\hb G}}
\nc{\Nn}{{\hb N}} \nc{\Pp}{{\hb P}} 
\nc{\Qq}{{\hb Q}} \nc{\Rr}{{\hb R}} \nc{\Zz}{{\hb Z}}
\nc{\mfa}{{\mf a}} \nc{\mfb}{{\mf b}} \nc{\mfk}{{\mf k}}
\nc{\mfm}{{\mf m}} \nc{\mfp}{{\mf p}} \nc{\mfq}{{\mf q}}
\nc{\mfr}{{\mf r}}
\nc{\fP}{{\mf P}}
\DMO*{\trdeg}{td}
\DMO*{\spec}{Spec}
\DMO*{\fork}{\nanchor}
\DMO*{\dnf}{\anchor}
\DMO{\RU}{RU}
\DMO{\deter}{det}
\DMO{\RM}{RM}
\DMO{\RC}{RC}
\DMO{\Real}{Re}
\DMO{\Imag}{Im}
\DMO{\tr}{tr}
\DMO{\qc}{QC}
\DMO{\Hu}{Hull}
\DMO{\leg}{length}
\DMO{\area}{area}
\DMO{\dia}{diameter}
\DMO{\iso}{Iso}
\DMO{\dis}{dist}
\DMO{\grad}{grad}
\DMO{\vol}{volume}
\DMO{\gra}{grad}
\DMO{\hd}{nbhd}
\DMO{\dv}{div}
\DMO{\Psl}{PSL}
\nc{\Mb}{\mathfrak^{2b/\delta}_{K_x}}
\nc{\Ma}{\mathfrak^{2a/\delta}_{K_x}}
\nc{\dif}{\mathrm{d}}
\nc{\G}{\Gamma}
\nc{\g}{\gamma}
\nc{\D}{\nabla}
\nc{\p}{\partial}
\nc{\DD}{\Delta^2}
\nc{\pp}{\partial^2} 
\nc{\de}{\delta}
\nc{\td}[2]{\trdeg{({#1}/{#2})}}
\nc{\dtd}[2]{\trdeg_{\delta}{({#1}/{#2})}}
\nc{\dspec}[1]{\spec_{\delta}{#1}}
\nc{\ddim}[1]{\dimen_{\delta}{#1}}
\nc{\gens}[1]{\langle {#1} \rangle}        
\nc{\gen}[2]{ {#1} \langle {#2} \rangle } 
\nc{\form}{\Omega}
\nc{\set}[1]{\left\{ {#1} \right\}}
\nc{\mr}{\hat}
\nc{\pr}{\partial}
\nc{\bc}[3]{\cB^{#1}({#2},{#3})=B^{#1}_{#2}(C^{#2}_{#3}(t)+B^{#1}_{#3}(C^{#2}_{#3}(t))} 
\nc{\tuple}[2]{{#1},\ldots,{#2}} \nc{\ptu}[2]{{#1}:\ldots:{#2}}

\nc{\maps}[3]{{#1}\!:\!{#2}\rightarrow{#3}}
\nc{\map}[2]{{#1}\rightarrow {#2}} \nc{\res}[2]{{#1} |_{#2}}
\nc{\imbed}{\hookrightarrow}
\title{On smooth moduli space of Riemann surfaces}

\author{Yong Hou\footnote{Supported by Ambore Mondell Fundation} }
\date{}

\affil{Institute for Advanced Study}

\begin{document}
\maketitle
\begin{abstract}

In this paper we study the smooth moduli space of closed Riemann surfaces. This smooth moduli is an infinite cover of the usual moduli space 
$\mathscr{M}_g$ of closed Riemann surfaces, and is identified with the Schottky space of rank $g.$ The main theorem of the paper is: Closed Riemann surfaces are uniformizable by Schottky groups of Hausdorff dimension less than one. This work seem to be the only paper in literature to study question of Riemann surface uniformization and its Hausdorff dimension. We develop new techniques of rational norm of homological marking of Riemann surface and,  
decomposition of probability measures to prove our result. As an application of our theorem we have existence of period matrix of Riemann surface in coordinates of 
smooth moduli space. 

\end{abstract}
\setcounter{tocdepth}{1}
\tableofcontents
\section{Introduction and Main Theorem}
The main theorem of this paper is:
\begin{thm}\label{main1}
Every closed Riemann surface can be uniformized by a Schottky group of Hausdorff dimension $<1.$
\end{thm} 

Throughout this paper $R_g$ denotes closed Riemann surface of genus $g.$ A Kleinian group $\G$ is, finitely generated and discrete subgroup of $\mbox{PSL}(2,\mathbb{C}).$ Denote by $\Lambda_\G,$ the limit set of $\G,$ which is minimal closed $\G$-invariant, no-where dense, non-discrete, perfect subset of 
$\overline{\mathbb{C}}.$  We say $\G$ is of Hausdorff dimension $\mathfrak{D}_\G$ if $\Lambda_\G$ have Hausdorff
dimension of $\mathfrak{D}_\G.$ The region of discontinuity of $\G$ is $\Omega_\G=\overline{\mathbb{C}}-\Lambda_\G.$
By Ahlfors theorem, $\Omega_\G/\G$ is of finite many components and each is of analytically finite type Riemann surface.\par
Rank-$g$ Schottky group $\G$ is, $g$-generated free, purely loxodromic Kleinian group. Equivalently, $\G$ is Schottky group if it is convex-cocompact representation of rank-$g$ free group $\mathbb{F}_g$ into $\mbox{PSL}(2,\mathbb{C}).$ In particular, $\mathbb{H}^3/\G$ is a handle-body and  $\Omega_\G/\G$ is a Riemann surface.\par
\emph{Uniformization} of $R_g$ by Schottky group $\G$, see next section for details is, $R_g=\Omega_\G/\G.$ Theorem \ref{main1} states that: 
there exists a Schottky group $\G$ such that: $R_g=\Omega_\G/\G$ and $\mathfrak{D}_\G<1.$\par
The main difficulty of Theorem \ref{main1} is the lack of relation of uniformization $\G$ and its Hausdorff dimension. In fact, the question of what possible Hausdorff 
dimensions of Kleinian groups can uniformize $R_g$ has not been studied at all before.\par

There are some interesting consequences of Theorem \ref{main1}, which we will address separately in subsequent woks. In this paper, we will state at the end an immediate simple consequence of Theorem \ref{main1} which implies existence of period matrix in smooth moduli coordinates.

\subsection*{Strategy of the proof of Theorem \ref{main1}}
First we derive a criteria for handle-body to be uniformizable by a classical Schottky group. To do so, we establish an new relation between Hausdorff dimension of
limit set of $\G$ with it's primitive elements. This is done by a generalization of paradoxically decomposing \cite{CS,Hou2} the probability measure on the limit set to its generators and integrate over all primitive bases. From this decomposition we derive an new family of probability measures supported locally about each generators. By using relations among this family of probability measures, we derive an lower bound on the growth of the \emph{mean primitive displacement} of generating elements of $\G.$ We show that the mean primitive displacement growth implies upper bounds on  Hausdorff dimension. The measure decomposition is done abstractly on Cayley graph of free group in Section 3. The growth estimate is done for handle-body and is given in Section 4.\par
Secondly, that given a Riemann surface $R_g,$ it is identified as the conformal boundary of infinity $\partial_\infty M$ of hyperbolic $3$-manifold $M=(\mathbb{H}^3\cup\Omega_\G)/\G.$ This conformal identification is through markings on $R_g.$ More precisely, It is classical known fact that \cite{AS, Bers}, we have a morphism $\phi$ from canonical basis with orientation of $H_1(R_g,\mathbb{Z})$ into Schottky space. This is done by representing half basis of a given basis by conformal maps of 
$\overline{\mathbb{C}}.$ Every half basis of a given canonical basis generates a subgroup of $H_1(R_g,\mathbb{Z}).$ The morphism establishes a injective corresponds of these subgroups into the Schottky space. We call a given half basis of a canonical basis a homological marking. However, it is completely unknown when a given point of Schottky space can actually cover a point of moduli space $\mathscr{M}_g.$\par
We study homological markings by defining a real function $Q$ on markings of $R_g.$ We call it \emph{rational norm} of a marking. This is defined through ratio of geodesic representative curves with the dual geodesic to the marking geodesic as the unique minimal curve. We also define a conformal invariant of $R_g,$ the  $Q$-spectrum of $R_g$ as the, collection of $Q$ values under variation of all markings of $R_g.$\par
We then establish an inequality Lemma \ref{mark-length}, between translation length of primitive elements of the image Schottky group to $Q.$ This is done through the use of theory of extremal length. Where we use interpolation of hyperbolic metrics on planar domains. \par
Next we show that there exists some marking such that 
$Q$ is bounded below by $\frac{2\lambda_g}{\pi}g\log(2g)$ for some $\lambda_g>2.$ We call a marking such that $Q$ satisfies this lower bound is \emph{positive}. This is done by study geodesic length ratios through elementary arcs, which are part of the marking curve on pants decompositions of $R_g.$ The existence proof of Lemma \ref{exist-[c]} is by contradiction. Under assumption, we show by computation, one can always choose elementary arcs to construct curves for $[c]$ so that $Q([c])$ achieve value greater than $\frac{2\lambda_g}{\pi}g\log(2g).$ These computations are done on pair of pants. \par
Next we show that there exists some marking such that $Q$ satisfies the lower bound inequality as in Lemma \ref{mark-length} under action of 
$stab(\phi)\subset Sp(2g,\mathbb{Z})$, stabilizer subgroup of $\phi.$ This is Proposition \ref{invariant}. The idea is that, by taken a marking provided by 
Lemma \ref{exist-[c]}, we can show either this marking stays positive under elementary matrices $E_{lm}$ of $stab(\phi)$ or, there must exists another marking that have larger value of $Q$ which will be positive under $E_{lm}.$ In fact, the ratio of lengths will increase under Dehn twist.\par
Proposition \ref{invariant} implies that there exists a marking such that all primitive elements of the Schottky group will have the desired mean displacement on the handle-body. This is show in Proposition \ref{add-length}. 
\par
Finally we use these marking estimates and result on handle-body Hausdorff dimension estimates to finish proof of the main theorem. As an immediate application
of Theorem \ref{main1}, we can explicitly express period matrix of points of $\mathscr{M}_g$ in coordinates of Schottky space.\par
\centerline{Acknowledgement}\par
This work is made possible by unwavering supports and brilliantly insightful conversations from Peter Sarnak, whom the author is greatly indebted to. I wish to express my deep gratitude and sincere appreciation  to Dave Gabai for continuous amazing support and many fundamentally influential conversations which greatly       impacted my works. I wish to express my deep gratitude to Ian Agol for many enlightening and brilliant brainstorming conversations. I also want to express my greatest appreciation to Benson farb for been able to support and help me mathematically under any circumstance without wavering. These people are critical to me to make this work possible. \par 
I especially want to express my greatest appreciation to Ian Agol for spending very long time with me to go through the details of the computations of the paper during my visit to MSRI at Berkeley. We corrected many typos in the previous draft and also a correction to a statement. Our discussion of this work over the days at MSRI are incredibility valuable.\par
This paper is dedicate to: Ying Zhou and her mom Junlan Wang.
\section{Covering of Moduli space $\mathscr{M}_g$}\label{cover}
We give here a very brief and basic introduction of $\mathfrak{J}_g$ as smooth moduli of a closed Riemann surface $R_g.$ See \cite{AS}, \cite{HJ}, \cite{FB}, \cite{MM},\cite{MS} for backgrounds materials and some of the details.\par
Schottky group $\G$ of rank $g$  is defined as convex-cocompact discrete faithful representation of $\mathbb{F}_g$ in $\text{PSL}(2,\mathbb{C}).$ It follows that 
$\G$ is freely generated by purely loxodromics $\{\g_i\}_1^g$. This implies we can find collection of disjointed 
closed topological disks $D_i,D_i', 1\le i\le g$ in the Riemann sphere $\partial\mathbb{H}^3=\overline{\mathbb{C}}$ with
boundary curves $\partial D_i=\Delta_i,\partial D_i'=\Delta_i'.$ By definition $\Delta_i$ are closed Jordan curves in Riemann sphere 
$\partial\mathbb{H}^3,$ such that $\g_i(\Delta_i)=\Delta'_i$ and $\g_i(D_i)\cap D'_i=\emptyset.$ Whenever there exists a
set $\{\g_i,...,\g_g\}$ of generators with all $\Delta_i,\Delta'_i$ as circles, then it is 
called a classical Schottky group with $\{\g_1,...,\g_g\}$ classical generators.\par
Schottky space $\mathfrak{J}_g$ is defined as space of all rank $g$ Schottky groups up to conjugacy by $\text{PSL}(2,\mathbb{C}).$ By normalization,
we can chart $\mathfrak{J}_g$ by $3g-3$ complex parameters. Hence $\mathfrak{J}_g$ is $3g-3$ dimensional complex manifold. The bihomolomorphic 
$\text{Auto}(\mathfrak{J}_g)$ group is $\text{Out}(\mathbb{F}_g),$ which is isomorphic to subgroup of the handle-body group. On the other hand, $\mathfrak{J}_{g,o}$ is not submanifold. In fact, it is nontrivial result due to Marden that $\mathfrak{J}_{g,o}$ is non-empty and non-dense set of $\mathfrak{J}.$ However, it follows from
a theorem of Hou \cite{Hou}, $\mathfrak{J}^\lambda_{g,o}$ is $3g-3$ dimensional complex submanifold. Here  $\mathfrak{J}^\lambda_{g,o}$ denotes space of classical Schottky groups of Hausdorff dimension $<\lambda.$\par
Let $\mathscr{T}_g$ and  $\text{Mod}_g$ be the Techimuller space of $R_g$ and it's mapping class group respectively. $\mathscr{T}_g$ is the universal cover of
$\mathfrak{J}_g.$ In fact, there exists subgroup $\text{Mod}^\phi_g\subset\text{Mod}_g$ which depends on a given symplectic-morphism 
$\phi: \pi_g\to\mathbb{F}_g.$ The dependence of $\phi$ is only up to conjugacy within $\text{Mod}_g.$ It follows $\text{Mod}^\phi_g$ is \emph{infinite index}, torsion-free subgroup of $\text{Mod}_g.$ Since $\mathscr{T}_g/\text{Mod}_g=\mathscr{M}_g,$ we have the following commutative diagram of holomorphic covers: 
\begin{displaymath}
\xymatrix{\mathscr{T}_g \ar[rr]^{\pi_U}\ar[rd]^{\pi_F}
& &\mathfrak{J}_g\ar[ld]^{\pi_S}\\
&\mathscr{M}_g}
\end{displaymath}
In particular, $\mathfrak{J}_g$ is infinite cover of $\mathscr{M}_g,$ hence can be considered \emph{smooth} moduli of $R_g.$
We end this section by restating Theorem \ref{main1}:
\begin{thm}\label{main2}
Let $[R_g]\in\mathscr{M}_g,$ there exist $\G\in\pi_S^{-1}([R_g])$ of  $\mathfrak{D}_\G<1.$
\end{thm}
\section{Probability Measures and $\|\mathbb{W}_\G\|_x$ on Schottky Space}
Let $\G$ be a finitely generated free group with generating set $\omega.$ We denote the collection of all generating sets by $W_\G.$ For Schottky group with 
chosen point
$x$ of hyperbolic space $\mathbb{H}^3,$ we also denote the minimal generator of $\omega$ with respect to $x$ by $\mathfrak{w}_x(\omega)=\sum_{\alpha\in\omega}\frac{\dis(\alpha x,x)}{g}.$ \par
We also denote the Cayley graph of $\G$ with symmetric generating set $S_\omega=\omega\cup \omega^{-1}$ by $C\G_{S_\omega}.$
\begin{defn}
For $\G$ Schottky group, we define:
\begin{itemize}
\item
$\mathbb{W}_\G$ the set of collection of all free basis $\omega$ of $\G,$
\item
$\|\mathbb{W}_\G\|_x$ of $\G$ at point $x$ as $\|\mathbb{W}_\G\|_x=\inf_{\omega\in\mathbb{W}_\G}\mathfrak{w}_x(\omega).$\par
We call $\|\mathbb{W}_\G\|_x$, the \emph{mean norm} of $\mathbb{W}_\G.$

\end{itemize}
\end{defn}
\begin{thm}\label{sys-gen}
Let $\G$ be a Schottky group of rank $g$ and Hausdorff dimension $\mathfrak{D}_\G.$  There exist nonatomic Borel measure $\sigma_x$ on 
$\Lambda_\G\times\mathbb{W}_\G$ of total mass $<1$ such that,
\[\mathfrak{D}_\G\le \frac{\log\left(\frac{g}{\sigma_x(\Lambda_\G\times\mathbb{W}_\G)}\right)}{\|\mathbb{W}_\G\|_x}.\] 
\end{thm}
Let $\G$ be Schottky group. Suppose $\omega_m=<\g_1,...,\g_g>$ is the basis of minimal translational length $T_{\g_j}$. Then we can estimate 
$\sigma_x(\Lambda_\G\times\mathbb{W}_\G)$ based on relations among elements of this minimal basis. 
\begin{cor}\label{sum-main}
Let $\G$ be a Schottky group. Suppose there exists $\lambda>0$ such that $T_{\g_j}\le\lambda T_{\g_i}$ for $\g_i,\g_j\in\omega_m.$
Then there exists $x\in\mathbb{H}^3$ such that,
\[\mathfrak{D}_\G\le \frac{(\lambda-1)\log(2)+(\lambda+1)\log(g)}{\|\mathbb{W}_\G\|_x}.\]
\end{cor}
Theorem \ref{sys-gen} is proved by constructing families of Borel probability measures on the limit set $\Lambda(\G).$ These probability measures are averages over
$W_\G$ of measures associated to elements of $W_\G.$ \par
We first state a decomposition lemma on $(C\G_{S},\mathscr{W})$, with respect to some chosen symmetric generating set $S_\omega$ and word metric $\mathscr{W}.$ Note that 
$(C\G_{S},\mathscr{W})$ is Gromov hyerbolic, or $\delta$-hyperbolic. \par Let $\mu_o$ be quasi-conformal measure, which is the Patterson-Sullivan measure \cite{CO}, on 
$(C\G_{S},\mathscr{W}).$ It is given by: 
\[ c_\G^{-1}\Psi_o^{\mathfrak{D}_\G}(\gamma^{-1}o,\zeta)\le\frac{d \gamma^*\mu_o}{d\mu_o}(\zeta)\le c_\G\Psi_o^{\mathfrak{D}_\G}(\gamma^{-1}o,\zeta), \quad\quad c_\G>0.\]
Where $\Psi_o^{\mathfrak{D}_\G}(\gamma^{-1}o,\zeta)=e^{-\mathfrak{D}_\G B_o(\gamma^{-1}o,\zeta)},$ with $B_o(\gamma^{-1}o,\zeta)$ the Buseman function.
We normalize so that $\mu_o$ is always taken to be probability measure on $\partial C\G.$ In addition, when we choose a indexing of $\mathbb{W}_\G$, we denote it by $\omega_j.$
\begin{lem}\label{measure}
There exists Borel measures $\sigma^+_o,\sigma^-_o$ on $\partial C\Gamma\times\mathbb{W}_\Gamma$ with $\sigma^+_o$ a probability measure and, for each $\omega\in\mathbb{W}_\G$ we have a family of Borel measures $\{\nu_{o,\g}\}_{\g\in\omega}$ and $\{\nu_{o,\g^{-1}}\}_{\g\in\omega}$ on 
$\partial C\G$ and $C_\omega\in (0,c_\G), c_\omega\in(c^{-1}_\G,c_\G)$ of the following properties:
\begin{itemize}
\item
\[ \eta^+_{o,\omega}=\sum_{\g\in\omega}\nu_{o,\g}, \eta^-_{o,\omega}=\sum_{\g\in\omega}\nu_{o,\g^{-1}},\]
$\eta^+_{o,\omega}$ is probability measures on $\partial C\G$ and $\eta^-_{o,\omega}(\partial C\G)<1,$
 for each $\omega\in\mathbb{W}_\G.$
\item
\[ \frac{d\sigma_o^+}{d\sigma_o^-}(\zeta,\omega)=\frac{d\eta^+_{o,\omega}}{d\eta^-_{o,\omega}}(\zeta).\] 
 \item
\[\frac{1}{(g-1)+C_\omega}\int_{\partial C\G}c_\omega\Psi^{\mathfrak{D}_\G}_o(\g^{-1}o,\zeta)d\nu_{o,\g^{-1}}=\int_{\partial C\G}d\nu_{o,\g}; \mbox{ for }  \g\in\omega.\]
\end{itemize}
\end{lem}

The proof of next lemma will be based on a generalization of Culler-Shalen paradoxical measure decompositions
for isometry groups of $\mathbb{H}^3$ \cite{Hou2}. These paradoxical decomposition were first done by Culler-Shalen \cite{CS} for the case of 
$2$-generated cocompact Kleinian group, and later generalized in \cite {Hou2} to all free finitely generated subgroups of $\text{PSL}(2,\mathbb{C}).$ 
They were interested in the embedded Margulis tube volume bounds. 
\begin{lem}\label{decomp}
There exists collection of Borel measures $\{\nu_{\g}\}_{\g\in\omega}$, $\{\nu_{o,\g^{-1}}\}_{\g\in\omega}$  on $\partial C\G$ such that, 
$\sum_{\g\in\omega}\nu_{o,\g}=\eta^+_{o,\omega}$ is probability measure and, $\sum_{\g\in\omega}\nu_{o,\g^{-1}}=\eta^-_{o,\omega}$ 
is measures of total mass $<1.$ In addition, there exists $c_\omega\in(c^{-1}_\G,c_\G)$ such that,
\[\sum_{\g\in\omega}\int_{\partial C\G}c_\omega\Psi^{\mathfrak{D}_\G}_o(\g^{-1}o,\zeta)d\nu_{o,\g^{-1}}=(g-1)+C_\omega .\]
Where $C_\omega\in (0,c_\G).$ In addition we have, 
\[\frac{1}{(g-1)+C_\omega}\int_{\partial C\G}c_\omega\Psi^{\mathfrak{D}_\G}_o(\g^{-1}o,\zeta)d\nu_{o,\g^{-1}}=\int_{\partial C\G}d\nu_{o,\g}.\]
\end{lem}
\begin{proof}
Let $\mathcal{S}=\omega\cup\omega^{-1}.$
Let us write every element $\g\in\G$ as a reduced word $w_1\dotsm w_n$
with $\{w_j\}\subset\mathcal{S}$. Then we have the decomposition of $\G$ as
$\G=\{1\}\coprod\coprod_{\g\in\mathcal{S}}I_\g$, where $I_\g$ is the set of
nontrivial elements in $\G$ with inital letter $\g$. By the fact that $\G$
act freely on $C\G$ we have $C\G=\G o=\{o\}\coprod\coprod_{\g\in\mathcal{S}}V_\g$
where $V_\g=\{w o: w\in I_\g\}$. Let $\cV$ denote the collection consisting of all sets of the
form $\coprod_{\g\in\mathcal{S}'}V_\g$ or $\{o\}\coprod\coprod_{\g\in\mathcal{S}'}V_\g$
for $\mathcal{S}'\subset\mathcal{S}$. \par
By Poincare series delta-mass construction,
\[\lim_{s\searrow\mathfrak{D}}\frac{\sum_{v\in V_\g}\exp(-s\dis(o,v))\delta_v}{\sum_{v\in V_\g}\exp(-s\dis(o,v))},\]
 with $C\G$ and $V_\g$, \cite{Hou2} Proposition $2.5$,
we get a family of Borel measures $(\mu_{y,{V_\g}})_{y\in C\G}$ for each
$\g\in\mathcal{S}$, and $\text{supp}(\mu_{y,{V_\g}})=\bar V_\g\cap\partial C\G.$ In general, we need to p Then, $\mu_{o,C\G}$ is a probability
measure on $\partial C\G.$ In fact, $\mu_{o,C\G}$ is the normalized Patterson-Sullivan measure centered at point $o.$ Define $\rho_\g:=\mu_{o,V_\g}$ 
for each $\g\in\mathcal{S}$. By the above decomposition of $C\G$, we have
$\mu_{o,C\G}=\mu_{o,{o}}+\sum_{\g\in\mathcal{S}}\rho_\g$. But $\mu_{o,{o}}=0.$ \par
Since $C\G=V_{\g^{-1}}\coprod \g^{-1} V_{\g}$, we have $\g^{-1} V_{\g}\in\cV$.  
Then by quasiconformal transformation property of Patterson-Sullivan measure, we get 
\[ \mu_{\g o,V_{\g}}=\g^{-1*}\mu_{o,C\G-V_{\g^{-1}}}=\g^{-1*}(\mu_o-\rho_{\g^{-1}}).\]
This implies we have,
\[{\dif}\mu_{\g o,V_{\g}}=c_{\g}\Psi^{\mathfrak{D}_\G}_o(\g^{-1}o,\xi){\dif}\mu_{o,V_{\g}},\mbox{ for some } c_{\g}\in(c_\G^{-1},c_\G).\]
From this, we get
\[\int_{\partial C\G}c_{\g}\Psi^{\mathfrak{D}_\G}_o(\g^{-1}o,\xi){\dif}\rho_{\g^{-1}}=\int_{\partial C\G}{\dif}(\g^{-1*}
(\mu_{o,C\G}-\rho_{\g}))=1-\int_{\partial C\G}{\dif}\rho_{\g}. \]
Since $\mu_{o,C\G}=\sum_{\g\in\omega}\rho_\g+\sum_{\g\in\omega}\rho_{\g^{-1}},$ we have one of $\sum_{\g\in\omega}\rho_\g, \sum_{\g\in\omega}\rho_{\g^{-1}}$
must have total weight over $\partial C\G$ of $\ge\frac{1}{2}.$ This is one of the property two that we exploit later in next corollary to show lower bounds of $\sigma_o^-$ based on indexing choice of $\mathbb{W}_\G.$ \par 
Set $C_\omega=\sum_{\g\in\omega}\int_{\partial C\G}{\dif}\rho_{o,\g^{-1}}.$ There exists $c_\omega\in (c^{-1}_\G, c_\G)$ such that,
\[\nu_{o,\g^{-1}}=\rho_{o,\g^{-1}},\]
\[\nu_{o,\g}=\frac{\mu_o-\rho_{o,\g}}{c_\omega(g-1)+c_\omega\sum_{\g\in\omega}\int_{\partial C\G}{\dif}\rho_{o,\g^{-1}}},\]
satisfies our conditions. 
\end{proof}

\begin{proof}{Proposition \ref{measure}:}\\
For $\omega\in\mathbb{W}_\G,$ define measures by $\eta^+_{o,\omega}=\sum_{\gamma\in\omega}\nu_{o,\gamma},$ and 
$\eta_{o,\omega}^-=\sum_{\g\in\omega}\nu_{o,\gamma^{-1}}.$ By Lemma \ref{decomp}, $\eta^{+}_{o,\omega}$ is a probability measures on 
$\partial C\G.$ 
\par Index $\mathbb{W}_\G=\cup_{i\ge 1}\omega_i.$
Set $\sigma^+_{o,N}=\frac{1}{N}\sum_{i=1}^N\eta_{o,\omega_i}^+\otimes\delta_{\omega_i},$  $\sigma^-_{o,N}=\frac{1}{N}\sum_{i=1}^N\eta_{o,\omega_i}^-\otimes\delta_{\omega_i}.$ 
We extend
$\sigma^{+}_{o,N},\sigma^-_{o,N}$ to $\partial C\G\times\mathbb{W}_\G$ trivially by setting it to be zero on $\mathbb{W}_\G-\cup_{j\le N}\omega_j.$
Note that since $\eta^+_{o,\omega_j}$ are probability measures, we have for all $N$,
\[ \int_{\partial C\G\times\mathbb{W}_\G}d\sigma^{+}_{o,N}=1.\]
Since $\eta^-_{o,\omega_j}(\partial C\G)< 1$ for all $j$ we have,
\begin{align*} 
\int_{\partial C\G}d\sigma^-_{o,N}<1.
\end{align*}
We have the weak-limit of, $\sigma^+_n\to\sigma^+_o$ and $\sigma_n\to\sigma_o.$ By construction we have $\sigma^\pm_o$ satisfies our conditions. 
\end{proof}
\begin{cor}\label{index}
There exists a sequence of Borel measures $\sigma^{\pm}_{o,n}$ on $\partial C\G\times\mathbb{W}_\G$ such that, support of $\sigma^{\pm}_{o,n}$ is
$\partial C\G\times\cup_{j\le n}\omega_j$ for a given index of $\mathbb{W}_\G$ and, $\sigma^{\pm}_{o,n}\to\sigma^{\pm}_o$ weakly. In addition,  $\sigma^+_o$ is a probability measure, so it's total weight is independent on index of 
$\mathbb{W}_\G,$ and there exists a indexing of $\mathbb{W}_\G$ such that, $\sigma^-_o(\partial C\G\times\mathbb{W}_\G)\ge\frac{1}{2}.$
\end{cor}
\begin{proof}
Since $\eta^+_{o,n}$ for all $n$ are probability measures, we can easily see by our construction in the proof of Proposition \ref{measure} that $\sigma^+_{o,n}\to \sigma_{o}$ is indexing independent of $\mathbb{W}_\G.$\par
We can choose a indexing of $\mathbb{W}_\G$ such that given $\omega_i$ we have either $\omega_{i+1},\omega_{i-1}$ is the inverse of $\omega_i.$
Since $\mu_o=\eta^-_{o,\omega_i}+\eta^-_{o,\omega^{-1}_i}$ we have by using this indexing, 
$\sigma^-_{o,n}(\partial C\G\times\mathbb{W}_\G)\to\frac{\mu_o}{2}(\partial C\G).$ Alternatively, if given a indexing 
such that $\sigma^-_o(\partial C\G\times\mathbb{W}_\G)<\frac{1}{2}$ then we can replace all $\omega_i$ with $\omega^{-1}_i.$ This would imply 
$\frac{1}{N}\sum_{i=1}^N\eta^-_{o,\omega_i}( \partial C\G\times\mathbb{W}_\G)\ge\frac{1}{2}$ for large $N.$ Hence we have required result. 
\end{proof}
\section{Hausdorff Dimension of Handlebody}
Take $\G$ to Schottky group. Then the parameter in Lemma \ref{decomp} is $c=1.$   \par
Let $\Lambda_\G$ be the limit set of $\G,$ which is the minimal invariant uniformly perfect closed no-where dense subset of $\mathbb{C}.$ The open invariant 
set $\Omega_\G=\mathbb{C}-\Lambda_\G$ is the region of discontinuity of $\G.$ Then 
$\mathfrak{D}_\G$ from previous section is the Hausdorff dimension of $\Lambda_\G,$ which we denote by $\mathfrak{D}_\G.$ We also have $\partial C\G$ is identified
with $\Lambda_\G.$ Since $\G$ is convex-cocompact, $C\G$ is quasi-isometric to $CH_\G$, convex hull of $\G.$
In addition, the density 
$\Psi_x(\gamma^{-1} x,\xi)$ is the Poisson kernel on $\mathbb{H}^3\cup\overline{\mathbb{C}}.$ Explicitly we have \cite{Hou2},
\[\Psi_x(\gamma^{-1} x,\xi)^{\mathfrak{D}_\G}=\left(\frac{1}{\cosh(\dis(x,\gamma x))-\sinh(\dis(x,\gamma x))\cos\angle \gamma^{-1} x x\xi}\right)^{\mathfrak{D}_\G}.\]
Here $\dis(x,y)$ is the hyperbolic distance for $x,y\in\mathbb{H}^3.$ Hence we can proceed explicit computation of measure over $\Lambda_\G.$\par
We summary this in the following:
\begin{cor}\label{SL}
Let $\G$ be Schottky group.  Lemma \ref{decomp} holds for $c=1$ and $\Psi^{-\mathfrak{D}_\G}_o(\gamma o,\zeta)$ given by the Poisson Kernel.
\end{cor}
We denote $\Psi_x(\g^{-1}x,0)$ for $\Psi_x(\gamma^{-1} x,\xi)$ when $\angle\gamma^{-1} x x\xi=0.$ This is simply $\exp(\dis(x,\g x)).$ \par 
Given $\alpha\in\omega$
we define, $\omega(\alpha)$ is the generating set given by $\alpha$ and $\{\g\alpha\}_{\g\in\omega-\alpha}.$ This is the generating set  given by \emph{shifting}
the original $\omega$ by $\alpha.$ \par
Next we will prove a lemma that will bound the mean displacement of any given $\omega$ at a point to that of the derivative of 
$\eta^-_{x,\omega(\alpha)}$ with respect to some transformed measure. This will provide connection between the mean displacement of $\omega$ with measures $\sigma_x^\pm.$

\begin{lem}[Shifting lemma]\label{shift}
Define $d\eta^*_{x,\omega}(\zeta)=\sum_{\g\in\omega}\Psi^{\mathfrak{D}_\G}_x(\g^{-1}x,\zeta)d\nu_{x,\g^{-1}}.$ Then there exists $\alpha\in\omega$ such that,
\[ \frac{d\eta^*_{x,\omega}}{d\eta^-_{x,\omega(\alpha)}}(\zeta)\ge\sum_{\g\in\omega}\frac{ \Psi^{\mathfrak{D}_\G}_x(\g^{-1}x,0) }{g} .\]

\end{lem}
\begin{rem}
Note that $\Psi^{\mathfrak{D}_\G}_x(\g^{-1}x,0)$ is notation for $\exp(\mathfrak{D}_\G\dis(\g^{-1}x,x)).$ \par 
Recall that Patterson-Sullivan measures provides exact quasiconformal distortions under group transformation. 
The idea of this lemma is that, we want to somehow gauge the distortion created when we change the measure $\nu_{x,\g}$ under transformations. More precisely, we want to estimate the distortion under Nielsen transformation. The point of shifting lemma is to estimate the lower bound of this distortion, which states that the distortion is at least the average of overall distortion.
\end{rem}
\begin{proof}
Let $\delta>0$ small. Define subset $E_\delta(x,\gamma)=\{\zeta\in\Lambda_\G|\angle{\gamma^{-1} x x\zeta}<\delta\}.$ Let $\alpha_\delta\in\omega$ such that
\[\nu_{x,\alpha^{-1}_\delta}(E_\delta(x,\alpha_\delta))\le\nu_{x,\g^{-1}}(E_\delta(x,\g)), \mbox{ for all } \g\in\omega. \]
Note that $\mbox{supp}(\nu_{x,\g^{-1}})=\bar V_{\g^{-1}}\cap\Lambda_\G.$ For $\beta\in\omega(\alpha_\delta)-\alpha_\delta$ we have, 
\[\mbox{supp}(\nu_{x,\beta^{-1}})=\mbox{supp}(\nu_{x,\alpha_\delta^{-1}\g^{-1}})\subset \mbox{supp}(\nu_{x,\alpha_\delta^{-1}}),\]
\[\quad\quad \nu_{x,\beta^{-1}}\le\nu_{x,\alpha^{-1}_\delta}. \]
This implies that for $\g'\in I_{\alpha^{-1}_\delta}$ (word start with $\alpha_\delta$) we have, 
\[ \nu_{x,\alpha_\delta^{-1}}(E_\delta(x,\g'))\ge\nu_{x,\beta^{-1}}(E_\delta(x,\g')), \mbox{ for all } \beta\in\omega(\alpha_\delta).\]
Hence we have,

\begin{align*}
\lim_{\delta\to 0}&\frac{\sum_{\gamma\in \omega}\int_{E_\delta(x,\g')}\Psi^{\mathfrak{D}_\G}_x(\gamma^{-1} x,\zeta)d\nu_{x,\gamma^{-1}}}
{\sum_{\gamma\in \omega(\alpha_\delta)}\int_{E_\delta(x,\g')}d\nu_{x,\gamma^{-1}}}\\
&\ge\lim_{\delta\to 0}\frac{\sum_{\gamma\in \omega}\inf_{\zeta\in E_\delta(x,\g')}\Psi^{\mathfrak{D}_\G}_x(\gamma^{-1} x,\zeta)\int_{E_\delta(x,\g')}d\nu_{x,\gamma^{-1}}}
{\sum_{\gamma\in \omega(\alpha_\delta)}\int_{E_\delta(x,\g')}d\nu_{x,\gamma^{-1}}}\\
&\ge\lim_{\delta\to 0}\frac{\sum_{\gamma\in \omega}\inf_{\zeta\in E_\delta(x,\g')}\Psi^{\mathfrak{D}_\G}_x(\gamma^{-1} x,\zeta)
\nu_{x,\alpha^{-1}_\delta}(E_\delta(x,\g'))}
{\sum_{\gamma\in \omega(\alpha_\delta)}\int_{E_\delta( x,\g' )}d\nu_{ x,\gamma^{-1}}}\\
&\ge\lim_{\delta\to 0}\frac{\sum_{\gamma\in \omega}\inf_{\zeta\in E_\delta(x,\g')}\Psi^{\mathfrak{D}_\G}_x(\gamma^{-1} x,\zeta)
\nu_{x,\alpha^{-1}_\delta}(E_\delta(x,\g'))}
{g\nu_{ x,\alpha_\delta^{-1}}(E_\delta( x,\g'))}\\
&\ge\sum_{\g\in\omega}\frac{ \Psi^{\mathfrak{D}_\G}_x(\g^{-1}x,0) }{g}
\end{align*}
\end{proof}

\begin{proof}{Theorem 3.2\text{}:}\par

Let $N>0$ be a large integer. By Lemma \ref{measure} and Corollary \ref{index}, there exists $\epsilon_N>0$ such that,
\begin{align*}
\int_{\Lambda_\G\times\mathbb{W}_\G}d\sigma^+_x &\ge(1-\epsilon_N)\sum_{1\le j\le N}\frac{1}{N}\int_{\Lambda_\G}d\eta^+_{x,\omega_j}\\
&\ge(1-\epsilon_N)\sum_{1\le j\le N}\sum_{\gamma\in \omega_j}\frac{1}{N}\int_{\Lambda_\G}d\nu_{x,\gamma^{+}}\\
&\ge(1-\epsilon_N)\sum_{1\le j\le N}\sum_{\gamma\in \omega_j}\frac{1}{N(g-1+C_{\omega_j})}\int_{\Lambda_\G}\Psi^{\mathfrak{D}_\G}_x(\gamma^{-1} x,\zeta)d\nu_{x,\gamma^{-1}}
\end{align*}
Since $c_\G=1$ and $C_{\omega_j}\in (0,1)$ for all $j$ we have,
\[\ge\frac{(1-\epsilon_N)}{g}\sum_{1\le j\le N}\sum_{\gamma\in \omega_j}\frac{1}{N}\int_{\Lambda_\G}\Psi^{\mathfrak{D}_\G}_x(\gamma^{-1} x,\zeta)d\nu_{x,\gamma^{-1}}.\]

By shifting Lemma \ref{shift}, for every $\omega_j$ there exists $\alpha_j\in\omega_j$ such that,

\[\frac{\sum_{\gamma\in \omega_j}\Psi^{\mathfrak{D}_\G}_x(\gamma^{-1} x,\zeta)d\nu_{x,\gamma^{-1}}}
{\sum_{\gamma\in \omega_j(\alpha_j)}d\nu_{x,\gamma^{-1}}}\ge\sum_{\gamma\in\omega_j}\frac{\Psi^{\mathfrak{D}_\G}_{ x}(\gamma^{-1} x,0)}{g}.
\]
Hence we have,
\begin{align*}
\int_{\Lambda_\G\times\mathbb{W}_\G}d\sigma^+_x \ge\frac{(1-\epsilon_N)}{g}\inf_{\omega\in \mathbb{W}_\G}\sum_{\g\in\omega}\frac{\Psi^{\mathfrak{D}_\G}_x(\gamma^{-1}x,0)}{g}
\sum_{1\le j\le N}\sum_{\gamma\in \omega_j(\alpha_j)}\frac{1}{N}\int_{\Lambda_\G}d\nu_{x,\gamma^{-1}}.
\end{align*}
Note the above inequality holds for any chosen indexing of $\mathbb{W}_\G$, and by Corollary \ref{index}, we can choose some indexing such that
$\sigma^-_x(\Lambda_\G\times\mathbb(W)_\G)\ge\frac{1}{2}.$ However, our shifting is indexing dependent hence we can't in general simply just bound by $1/2,$  unless we have some bounds on generators translation length.
\begin{align*}
\int_{\Lambda_\G\times\mathbb{W}_\G}d\sigma^+_x &\ge \frac{(1-\epsilon_N)}{g}\inf_{\omega\in \mathbb{W}_\G}\sum_{\g\in\omega}\frac{\Psi^{\mathfrak{D}_\G}_x(\gamma^{-1}x,0)}{g}
\sum_{1\le i\le N}\sum_{\gamma\in \omega_{j_i}(\alpha_{j_i})}\frac{1}{N}\int_{\Lambda_\G}d\nu_{x,\gamma^{-1}}\\
&\ge \frac{(1-\epsilon_N)}{g}\inf_{\omega\in \mathbb{W}_\G}\sum_{\g\in\omega}\frac{\Psi^{\mathfrak{D}_\G}_x(\gamma^{-1}x,0)}{g}
\sum_{1\le i\le N}\frac{1}{N}\int_{\Lambda_\G}d\eta^-_{x,\omega_{j_i}(\alpha_{j_i})}.\\
 \mbox{Since  } &\Psi^{\mathfrak{D}_\G}_x(\g^{-1},0)=\exp(\mathfrak{D}_\G\dis(x,\g^{-1} x)),\\
&\ge \frac{(1-\epsilon'_N)}{g}\inf_{\omega\in\mathbb{W}_\G}\sum_{\g\in\omega}\frac{e^{\mathfrak{D}_\G\dis(x,\g^{-1}x)}}{g}\int_{\Lambda_\G\times\mathbb{W}_\G}d\sigma^-_x.
\end{align*}
By $\sigma_x^+(\Lambda_\G\times\mathbb{W}_\G)=1$ and  $\sigma^-_x(\Lambda_\G\times\mathbb{W}_\G)$ we have, 
\[1\ge  \frac{(1-\epsilon'_N)\sigma^-_x(\Lambda_\G\times\mathbb{W}_\G)}{g}\inf_{\omega\in\mathbb{W}_\G}\sum_{\g\in\omega}\frac{e^{\mathfrak{D}_\G\dis(x,\g^{-1}x)}}{g}.\]
This implies that,
\[\inf_{\omega\in\mathbb{W}_\G}\sum_{\g\in\omega}\frac{e^{\mathfrak{D}_\G\dis(x,\g^{-1}x)}}{g}\le \frac{g}{(1-\epsilon'_N)\sigma^-_x(\Lambda_\G\times\mathbb{W}_\G)}.\]
From the inequality,
\[\inf_{\omega\in\mathbb{W}_\G}\sum_{\g\in\omega}\frac{\mathfrak{D}_\G\dis(x,\g^{-1}x)}{g}
\le\inf_{\omega\in\mathbb{W}_\G}\log(\sum_{\g\in\omega}\frac{e^{\mathfrak{D}_\G\dis(x,\g^{-1}x)}}{g}),\]
and since $\epsilon_N'$ is arbitrarily small, hence we have,
\[\mathfrak{D}_\G\le \frac{\log\left(\frac{g}{\sigma^-_x(\Lambda_\G\times\mathbb{W}_\G)}\right)}{\|\mathbb{W}_\G\|_x}.\]
\end{proof}
To give \emph{best} possible estimate of $\sigma^-_x(\Lambda_\G\times\mathbb{W}_\G)$ we need some control of the shifting of $\mathbb{W}_\G.$ This shifting information  is provided by the the relationship among the generators. The idea is to pick the best possible indexing of $\mathbb{W}_\G$ such that the shifting will
remain sufficiently bounded below along the indexing.
\begin{cor}\label{2-main}
Let $\G$ be a Schottky group such that there exists $x\in\mathbb{H}^3$ and indexing of $\mathbb{W}_\G=\cup_j\{\omega_j\}$ such that
all shifting $\omega_j(\alpha_j)$ have $\nu_{x,\g^{-1}}\ge\frac{1}{2g}$ for $\omega_j(\alpha_j)$ then, 
\[\mathfrak{D}_\G\le \frac{2\log(g)}{\|\mathbb{W}_\G\|_x}.\]
\end{cor}
We can obtain an similar estimate with average bounds on $\nu_{x,\g^{-1}}$ as follows.
\begin{cor}\label{D-main}
Let $\G$ be a Schottky group such that there exists $x\in\mathbb{H}^3$ and indexing of $\mathbb{W}_\G=\cup_j\{\omega_j\}$ with  
$\nu_{x,\g^{-1}}(\Lambda_\G)\ge\frac{1}{g-1}\sum_{\beta\in\omega_j-\g}\nu_{x,\beta^{-1}}(\Lambda_\G)$ for every $\g\in\omega_j.$ Then 
\[\mathfrak{D}_\G\le \frac{2\log g}{\|\mathbb{W}_\G\|_x}.\]
\end{cor}
Let $\G$ be Schottky group. Suppose $\omega_m=<\g_1,...,\g_g>$ is the basis of minimal translational length $T_{\g_j}$. Then we can estimate $\nu_{x,\omega_j}$ based on relations among elements of this minimal basis. 
\begin{cor}\label{sum-main}
Let $\G$ be a Schottky group. Suppose there exists $\lambda\ge 1$ such that $T_{\g_j}\le\lambda T_{\g_i}$ for $\g_i,\g_j\in\omega_m.$
Then there exists $x\in\mathbb{H}^3$ such that,
\[\mathfrak{D}_\G\le \frac{(\lambda-1)\log(2)+(\lambda+1)\log(g)}{\|\mathbb{W}_\G\|_x}.\]
\end{cor}
\begin{proof}
For convex cocompact $\G,$ Patterson-Sullivan measure is the unique $\mathfrak{D}_\G$-Hausdorff measure. This implies $\nu_{x,g^{-1}}, \g\in\omega$ is absolutely 
continuous to the Hausdorff measure of support $\bar V_{g^{-1}}\cap\Lambda_\G.$
By $T_{\g_j}\le\lambda T_{\g_i},$  we can choose
$x\in\mathbb{H}^3$ such that
$\nu_{x,\g^{-1}_k}(\Lambda_\G)\ge\frac{1}{2g}$ for some $\g_k\in\omega_m,$ and the total mass of $\nu_{x,\g^{-1}_j}$ is bounded by  $\nu_{x,\g^{-1}_j}(\Lambda_\G)\ge\nu^\lambda_{x,\g^{-1}_i}(\Lambda_\G).$ 
Let $\omega_m(\g_j)=\{\g_j,\g_i\g_j\}_{i\not=j}$ be the shifted basis. Since $\sum_{\g\in\omega_m(\g_j)-\g_j}\nu_{x,\g^{-1}}(\Lambda_\G)=\nu_{x,\g^{-1}_j}(\Lambda_\G)$ we have,
\begin{align*}
\sum_{\g\in\omega_m(\g_j)}\nu_{x,\g^{-1}}(\Lambda_\G)&\ge \nu^\lambda_{x,\g^{-1}_k}(\Lambda_\G)+\sum_{\g\in\omega_m(\g_j)-\g_j}\nu_{x,\g^{-1}}(\Lambda_\G)\\
&\ge \nu^\lambda_{x,\g^{-1}_k}(\Lambda_\G)+\nu^\lambda_{x,\g^{-1}_k}(\Lambda_\G)\ge\frac{2}{(2g)^\lambda}.
\end{align*}
Using the notations in the proof of Theorem \ref{sys-gen}, since $\omega_m$ is the minimal length generator basis, we can choose an indexing of 
$\mathbb{W}_\G$ such that,
\[\sum_{1\le i\le N}\frac{1}{N}\sum_{\g\in\omega_{j_i}(\alpha_{j_i})}\nu_{x,\g^{-1}}(\Lambda_\G)\ge \frac{2}{(2g)^\lambda}.\]
This implies that, $\sigma_x(\Lambda_\G\times\mathbb{W}_\G)\ge\frac{2}{(2g)^\lambda}$ which give the result.
\end{proof}
Rather than uniform bound of $\lambda$, we can relax the condition of Corollary \ref{sum-main} between $T_i,T_j.$ Arrange $T_{\g_i}\le T_{\g_j}$ for $i\le j$ as follows:
Suppose $T_j\le\lambda_{ji}T_i$ for some collection of
$\lambda_{ji}\ge 1, i\le j$.  We take convention that $\lambda_{ij}=1$ for
$i\le j.$ If $\g_j$ is the shifting generator for $\omega_m(\g_j)$, then by replacing generator $\g_k$ for, $k=j-1$ if $j>1$ and $k=2$ when $j=1$, with its inverse in $\omega_m$ if necessary, we can assume that $\nu_{x,\g^{-1}_k}(\Lambda_\G)\ge\frac{1}{2g}.$ This implies that, $\sum_{\g\in\omega_m(\g_j)}\nu_{x,\g^{-1}}(\Lambda_\G)\ge\frac{2}{(2g)^{\lambda'_{jk}}}.$  Set $\bar\lambda=\sup_{jk}\lambda'_{jk}$ we have the following version of above 
corollary:
\begin{cor}\label{sum-main-2}
Let $\G$ be a Schottky group. Suppose there exists $\lambda_{ij}\ge 1$ such that $T_{\g_j}\le\lambda_{ji} T_{\g_i}, i\le j$ for $\g_i,\g_j\in\omega_m.$
Then there exists $x\in\mathbb{H}^3$ such that,
\[\mathfrak{D}_\G\le \frac{(\bar\lambda-1)\log(2)+(\bar\lambda+1)\log(g)}{\|\mathbb{W}_\G\|_x}.\]
\end{cor}
Finally, we mention couple of interesting consequences of Theorem \ref{sys-gen} related to injectivity radius and classical Schottky groups.\par
Define $H_c=\sup\{\lambda| \mbox{ such that all Schottky group of  } \mathfrak{D}_\G<\lambda \mbox{ is classical}\}.$ 
$H_c$ is the maximal parameter such that if $\G$ have Hausdorff dimension $<H_c$ then $\G$ is classical Schottky group. It follows from theorem of Hou \cite{Hou}, such that $H_c$ exists. We have next obvious corollary:
\begin{cor}\label{class-gen}
There exists $\tau_c>0$ such that any Schottky group $\G$ of rank $g$ with 
$\|\mathbb{W}_\G\|_x>\tau_c\log\left(\frac{g}{\sigma^-_x(\Lambda_\G\times\mathbb{W}_\G)}\right)$ for all $x\in\mathbb{H}^3$ is classical Schottky group.
\end{cor}
\begin{proof}
By a theorem of Hou \cite{Hou}, there exists maximal $H_c>0$ such that any finitely generated free Kleinian $\G$ with Hausdorff dimension $<H_c$ is classical Schottky group. Set $\tau=\frac{1}{H_c}$ completes the proof.
\end{proof}
\begin{cor}  
Let $H=\mathbb{H}^3/\Gamma$ be hyperbolic handlebody of rank $g>1$ such that $\sigma^-_x(\Lambda_\G\times\mathbb{W}_\G))\ge 1/2.$
There exists universal $\tau_c>0$ such that, if 
the injectivity radius $i_H$ of $H$ satisfies $i_H>\tau_c\log(2g),$ then H is uniformized by classical Schottky group $\Gamma.$
\end{cor}
For a hyperbolic $3$-manifold $M$, if the $\pi_1(M)$ is not free then there is imbedded surface which would put an upper bounds on the injectivity radius. Hence
for sufficiently large injectivity of a given hyperbolic $3-$manifold, the fundamental group must be free. So we have the following corollary:
\begin{cor}
Let $M$ be a hyperbolic $3$-manifold with fundamental group of rank $g.$ If $i_M>c\log\left(\frac{g}{\sigma^-_x(\Lambda_\G\times\mathbb{W}_\G)}\right),$ then $M$ is unformized by a classical Schottky group.
\end{cor}

\section{Rational Norm $Q$ of Homological Markings on $R_g$ and extremal length} 

Let $H_1(R_g,\mathbb{Z})$ be the first homology group and denote by $B_1$ the set of canonical basis of $H_1(R_g,\mathbb{Z}).$
This is given by class: $[\alpha_i],[\beta_i]$ satisfies $<[\alpha_i],[\alpha_j]>=<[\beta_i],[\beta_j]>=0$ and $<[\alpha_i],[\beta_j]>=\delta_{ij},$
$<[\beta_j],[\alpha_i]>=-\delta_{ij}.$\par
Let $\pi:B_1\to BH_1$ be the projection to the collection of first set of $g$ cycles $\{\alpha_i\}_1^g.$ i.e. $<\alpha_i,\alpha_j>=0,$ half-basis.  Denote $B_\alpha$ the subgroup of $H_1(R_g,\mathbb{Z})$ generated by $[c]\in BH_1.$ We define $B_{1/2}$ to be the collection of all such subgroups $B_\alpha$ for $[c]\in BH_1.$
\par
There exists a $\phi: BH_1\longrightarrow \mathfrak{J}_g$ maps into the Schottky space. The map $\phi$ is a morphism such that maps all $\phi(\alpha_j)=\gamma_j, 1\le j\le g.$ The curves $\{\alpha_1,...,\alpha_g\}$
are called \emph{cut system}.  Denote $\G_{[c]}=<\g_1,...,\g_g>$ the image Schottky group of $[c]$ under $\phi.$ Each $B_\alpha$, subgroup of $H_1(R_g,\mathbb{Z})$ generated by $\{\alpha_1,...,\alpha_g\}$ uniquely determines $\G_{[c]}.$ Different set
of $\alpha'_i$ which generates same subgroup of  $H_1(R_g,\mathbb{Z})$ gives same Schottky group under $\phi,$ of different set of generators, which corresponds to $\alpha'_i$ cut system. We have injective map of $B_{1/2}$ into $\mathfrak{J}_g.$ The fundamental domain of $\G_{[c]}$ is conformally equivalent to the planar (genus zero) domain, $R_g-\cup_{1\le i\le g} \alpha_i.$ For details, see \cite{AS,Bers}. \par
   We denote the region of discontinuity of the image Schottky group $\G_{[c]}$ by $\Omega_{[c]}.$ The domain $\Omega_{[c]}$ is hyperbolic planar domain with hyperbolic metric $\rho_{[c]},$ which is the $\G_{[c]}$-invariant Poincare metric of hyperbolic disk. If $\rho_R$ is the Poincare hyperbolic metric of $R_g$ then, we have holomorphic covering map, $\pi_s: (\Omega_{[c]},\rho_{[c]})\to (R_g,\rho_{R}).$
\par
$B_1$ is invariant under the symplectic group $Sp(2g,\mathbb{Z}).$ Let $stab(\phi)$ denote the normal subgroup of 
stabilizer of $B_\alpha.$
The subgroup $stab(\phi)$ is generated by elements which corresponds to Nielsen transformations of generators of $\G_{[c]}.$ \par
Let $\theta\in Sp(2g,\mathbb{Z}),$ we define the action $\theta[c]$ for $[c]\in BH_1$ as: $\pi\theta\pi^{-1}[c].$
We also set $\hat BH_1=BH_1/stab(\phi).$ Denote elements of $\hat BH_1$ by $[\![c]\!].$\par
Given $[c]=\{\alpha_i\}_1^g\in BH_1,$ the collection of $g$ cycles $\{\sigma_j\}_1^g\in\pi^{-1}[c]$ 
are called \emph{dual cycles} of $\alpha_i$ cycles. Denote this collection by $D(\alpha).$
We also denote by $D(\alpha_i)$ to be the collection of all dual cycles to $\alpha_i$ in $D(\alpha).$ i.e. simple closed curves in $D(\alpha)$ of intersection $1$ with $\alpha_i$

For any given curve $\sigma \subset (R_g,\rho_R),$ we denote the hyperbolic length of $\sigma$ by $\ell(\sigma).$ This is also the hyperbolic length in $\pi^{-1}_s(\sigma)\subset (\Omega_{\G_{[c]}},\rho_{[c]}).$\par
For $[c]\in BH_1,$  Denote $[c]_i$ as $\alpha_i,$ the $i$-th cut system cycles. 
$\ell({[c]_i})$ is the hyperbolic length of the \emph{geodesic} representative curve of $\alpha_i.$ Since $R_g$ is compact, there exists a unique $\beta_i^*\in D(\alpha_i),$
geodesic representative curve such that,
$\ell(\beta^*_{\alpha_i})=\inf_{\beta\in D(\alpha_i)}\ell(\beta).$
We also define the following notations: \par
\[ \|[c]_{i}\|=\ell^2([c]_{i}), \quad\quad \|[c]_i\|_{D}=\ell^2(\beta^*_{\alpha_i}). \]
\begin{rem}
We make convention that for a $[c]\in BH_1$ represented by geodesics $\{\alpha_i\}_1^g,$ we identify $[c]$ to a basis in $H_1(R_g,\mathbb{Z})$ by adjoining the $\beta^*_{\alpha_i}$ the unique geodesic cycles as: $\alpha_i,\beta^*_{\alpha_i}.$ So we speak of $[c]$ as basis and element of $BH_1$ interchangeably through this identification. 
\end{rem}
\begin{defn}
Let $[c]\in BH_1$. We also define the  Rational Norms $Q([c]_i)$ of $[c]$ as the collection of all $Q([c])$:
\[ Q([c]_i)= \frac{\|[c]_i\|_{D}}{\|[c]_i\|},\quad Q([c])=\sum_{i=1}^{g}Q([c]_i).\]
\end{defn}
\begin{defn}
We define $Q_\lambda$ the $Q$-spectrum of $R_g$ as: 
\[Q_\lambda(R_g)=\{Q([c])\quad|\quad [c]\in BH_1\}. \]
\end{defn}
Note $Q_\lambda(R_g)$ is non-discrete countable set. Next proposition is obvious:
\begin{prop}\label{spectrum}
$Q_\lambda(R_g)$ is conformal invariant and defines a set-valued function on $\mathscr{M}_g.$
\end{prop}
\par
Let $\{T_{[c],i}\}^g_1$ denote the collection of translation length of elements $\{\g_i\}\in\G_{[c]}.$  
Set $\mathscr{T}_{[c]}=\frac{1}{g}\sum_{1\le i\le g}T_{[c]_i}.$ We call $\mathscr{T}_{[c]}$ the \emph{Schottky length} of $[c].$ \par

Given $R_g$ a Riemann surface or domain of $\mathbb{C}$. Denote $\text{conf}(R_g)$ space of all conformal metric on $R_g.$
Locally, $\psi\in\text{conf}(R_g)$ is given by quadratic differential $\psi(z)dz^2.$\par
Let $\Phi$ be a collection of curves in $R_g.$ Recall the \emph{extremal length}
$\mathscr{E}_{R_g}(\Phi)$ is given by:
\[\mathscr{E}_{R_g}(\Phi)=\sup_{\psi\in\text{conf}(R_g)}\frac{\inf_{\sigma\in\Phi}\left(\int_{\sigma}|\psi|\right)^2}{\int_{R_g}|\psi|^2}\]
Note that it is simple fact that $\mathscr{E}_{R_g}(\Phi)$ is conformal invariant. \par
Next we will use extremal length to establish lower bounds of $T_{[c]_i}$ by $Q([c]_i).$
\begin{lem}\label{mark-length}
\[T_{[c]_i}\ge\frac{\pi}{2}Q([c]_i). \]
\end{lem}
\begin{proof}
Let $C_i,C_i'\subset\mathbb{C}$ be the lift of $[c]_i$ which are Jordan curves that bounds disjoint closed disks $D_i,D_i'.$ such such: 
$\g_i(C_i)=C_i',$ $\g_i(D^o_i))\cap D^c_i.$ Here $\g_i$ is generators of $\G_{[c]}$ given by $\phi([c]).$ As before, $\rho_{[c]}$ is the hyperbolic metric on
$\Omega_{[c]}.$ \par
Let $\Phi_i$ be the collection of all paths connecting $C_i$ to $C_i'$ in $\mathbb{C}.$ Let $\mathscr{R}_i=\mathbb{C}-(D_i\cup D_i').$ Also set
$\mathscr{R}=\mathbb{C}-\cup_i(D_i\cup D_i'),$ and $D_i^c=\cup_{j\not=i}(D_j\cup D_j').$ 
\par
Note that since $\mathscr{R}\subset\Omega_{[c]},$ $\rho_{[c]}$ defines a hyperbolic metric on $\mathscr{R}.$
Let $w$ be any curve in $\mathscr{R}$ connecting $D_i,D_i'.$ Let $\ell(w)$ be the $\rho_{[c]}$-hyperbolic length of $w.$   \par
Denote $U$ collection of all curves in $\mathscr{R}_i$ connecting $D_i,D_i'.$ Choose a conformal metric $h_w$ on $\mathscr{R}_i$ such that, 
$\inf_{u\in U}\ell_{h_w}(u)>\ell(w).$ Let $\epsilon>0$ such that, the $D_\epsilon,$ $\epsilon$-neighborhood of $D^c_i$ is: 
$D_\epsilon\cap(D_i\cup D_i')=\emptyset.$
Choose a $\sigma_\epsilon(z)$ smooth function of $\mathbb{C},$ which is approximate characteristic cut-off function such that:
\[ \sigma_\epsilon(z)=\begin{cases}1 &\mbox{if } z\in\mathscr{R}\\
0 & \mbox{if } x\in D_\epsilon 
\end{cases}
\]
and, $d\rho_{\epsilon,w}^2=(d\rho_{[c]}^2)^{\sigma_\epsilon}(dh_w^2)^{1-\sigma_\epsilon}$ is of negatively pinched curvature. If we denote metric density by the same notation and write hyperbolic metric in conformal factors $e^{\rho_{[c]}}|dz|$ and $e^{h_w}|dz|$ we have,  $\rho_{\epsilon,w}=\sigma_\epsilon\rho_{[c]}+(1-\sigma_\epsilon)h_w$ is $\epsilon$-family of pinched negatively curved metric on $\mathscr{R}_i.$
\par 
First we establish bounds of $Q([c]_i)$ by extremal length $\mathscr{E}_{\mathscr{R}_i}(\Phi).$\par
It follows from the isoperimetric inequality for negatively pinched manifold \cite{APR,MJ} we have,
\[\int_{\mathscr{R}_i}|\rho_{\epsilon,w}|^2\le \left(\int_{C_i\cup C'_i}|\rho_{\epsilon,w}|\right)^2 .\] 
\[
\mathscr{E}_{\mathscr{R}_i}(\Phi_i)=\sup_{\psi\in\text{conf}(\mathscr{R}_i)}\frac{\inf_{\sigma\in\Phi_i}\left(\int_{\sigma}|\psi|\right)^2}{\int_{\mathscr{R}_i}|\psi|^2}
\ge \frac{\inf_{\sigma\in\Phi_i}\left(\int_{\sigma}|\rho_{\epsilon,w}|\right)^2}{\left(\int_{\mathscr{R}_i}|\rho_{\epsilon,w}|^2\right)}.\]
By isoperimetric inequality we have,
\[\ge \frac{\inf_{\sigma\in\Phi_i}\left(\int_{\sigma}|\rho_{\epsilon,w}|\right)^2}{\left(\int_{C_i\cup C'_i}|\rho_{\epsilon,w}|\right)^2}.\]
Since assuming $\epsilon$ is sufficiently small we have, $\int_{C_i\cup C'_i}|\rho_{\epsilon,w}|<\int_{C_i\cup C'_i}|\rho_{[c]}|+\delta_\epsilon,$ with 
$\delta_\epsilon\to 0.$ We have,
\[>  \frac{\inf_{\sigma\in\Phi_i}\left(\int_{\sigma}|\rho_{\epsilon,w}|\right)^2}{\left(\int_{C_i\cup C'_i}|\rho_{[c]}|+\delta_\epsilon\right)^2}.\]
By $\inf_{u\in U}\ell_{h_w}(u)>\ell(w),$ implies all curves $u$ in $\mathscr{R}_i$ connecting $D_i,D'_i$ that intersects $D^c_i$ must have 
$\ell_{\rho_{\epsilon,w}}(u)>\ell(w).$
Let $V$ denote curves in $\mathscr{R}$ connecting $D_i,D_i'.$ Since $w$ is a curve in $\mathscr{R},$ hence we have, 
\[\inf_{\sigma\in\Phi_i}\int_{\sigma}|\rho_{\epsilon,w}|\ge \inf_{v\in V_i}\int_{v}|\rho_{\epsilon,w}|\]
Since for $v\subset\mathscr{R}$ we have $\int_{v}|\rho_{\epsilon,w}|<\int_{v}|\rho_{[c]}|-\delta'_\epsilon$ for sufficiently small $\epsilon.$ This implies,

\begin{align*}
\mathscr{E}_{\mathscr{R}_i}(\Phi_i)&>\frac{ \inf_{v\in V_i}\left(\int_{v}|\rho_{[c]}|-\delta_\epsilon'\right)^2}{\left(\int_{C_i\cup C'_i}|\rho_{[c]}|+\delta_\epsilon\right)^2}
\ge\frac{ \inf_{v\in V_i}\left(\int_{v}|\rho_{[c]}|-\delta_\epsilon'\right)^2}{\left(2\int_{C_i}|\rho_{[c]}|+\delta_\epsilon\right)^2}\\
&>\frac{\inf_{\beta\in D(\alpha_i)}\ell^2({\beta})-\delta_\epsilon''}{4\ell^2([c]_{i})+\delta'''_\epsilon}.
\end{align*}
Since $\delta''_\epsilon,\delta_\epsilon'''$ can be made arbitrarily small by choose $\epsilon$ sufficiently small, the last inequality implies,
\[\mathscr{E}_{\mathscr{R}_i}(\Phi_i)\ge\frac{1}{4}Q([c]_i).\]
Let $g$ be the Mobius transformation so that $g\g_ig^{-1}$ of fixed points $0,\infty.$ We have $g(\mathscr{R}_i)=A_i$ is annulus centered at origin of radii $r_2>r_1.$ 
Since, 
\[g^*\mathscr{E}_{\mathscr{R}_i}(\Phi_i)=\mathscr{E}_{A_i}(g(\Phi_i))=\frac{1}{2\pi}\log(\frac{r_2}{r_1}).\]
Also note that the translation length of $g\g_ig$ is $\log(\frac{r_2}{r_1}),$ i.e: $T_{[c]_i}=\log(\frac{r_2}{r_1})$. Hence by conformal invariance of 
$\mathscr{E}_{\mathscr{R}_i}(\Phi_i)$ we have, $T_{[c]_i}\ge\frac{\pi}{2}Q([c]_i).$ 

\end{proof}

\section{Pants decomposition and bound of $Q([c])$}
Next we show the existence of homological basis which gives some lower bounds for the Rational Norm $Q([c]).$ 
\begin{lem}\label{exist-[c]}
There exists $[c]\in BH_1$ such that
\[Q([c])>\frac{2\lambda_g}{\pi}g\log(g),\]
for some $\lambda_g>2,$ when $g=2$ and $\lambda_g>3$ if $g>2.$
\end{lem}
\begin{proof}
We will choose $[c]$ which is of relatively short length to it's $\beta^*_i$ by compare arcs on pair of pants.  The case $g=1$ is trivial. Assume $g=2.$ For $c\in [c],$ denote by $c=\{\alpha_1,\alpha_2\}$ with $\alpha_i$ the non-separating curves and it's dual curves (intersection $<\alpha_i,\beta_j>=\delta_{ij}$) by $\beta_i.$ \par
Since $g=2,$ take $\{\alpha_1,\alpha_2,\beta_1,\beta_2\}$ then, $\frac{\ell(\alpha_i)}{\ell(\beta_i)}<1$ or inverse is $<1$ for $1\le i\le 2.$ So we can always choose 
$c$ so that $\frac{\ell(\beta_i)}{\ell(\alpha_i)}>1$ for $1\le i\le 2.$ Hence we have $\sum Q([c]_i)>2$ In addition, we have some $\lambda_g>2.$
 \par
For $g>2,$ we need to decompose $R_g$ into $2g-2$ pair of pants and estimate $\ell(\beta)/\ell(\alpha)$ on pant components. \par
Given $[c]\in BH_1$ we complete $[c]$ with separating curves and let $P=\{P^k\}, 1\le k\le 2g-2$ denote the associated pants decomposition. For each $P^k\in P$ 
we cut into two hexagonal pieces and we mark it by border geodesic arcs by $B_k=\{g_k,b_k,c_k,c'_k,e_k,f_k,g'_k\}.$ The arcs are:  $a_k$ right geodesic arc; $b_k$top connector geodesic arc; $c_k$ top left geodesic arc, $d_k$ top right left geodesic arc; $e_k$ middle connector geodesic arc; $f_k$ bottom geodesic arc; 
$g_k$ left geodesic arc.
Conversely, given any $2g-2$ pair of pants decomposition of $R_g$ there exists homological
basis $\{\alpha_i,\beta_i\}, 1\le i\le g$ which are not separating curves. To compare homological length we use idea of elementary arcs in \cite{BSE}. A elementary arc 
$e$ is a arc on $P^k$ with end points lie on the boundary of $P^k$ such that it intersect border geodesic arcs $B_k$ at most two points in the interior of $P^k$. By definition all border geodesic arcs in $B_k$ are elementary arcs. \par 
The homological curve that we will be looking for must have minimal, although not necessarily zero Dehn twist, since by triangle inequality one can always shorten a curve by reduce its twist. We want to
able to show that there exists $[c]$ such that $\alpha_i$,  the cut curves can be made successively relatively short compare to its, dual $\beta_i$ curves.\par
The idea is that, if there  are no such $[c]$ exists on $R_g$ then we will have an contradiction with the hyperbolicity of $R_g$. This contradiction is reached through
length computations of arcs on $P^k.$  To do so, 
we will compute relative $\frac{\ell(e)}{\ell(e')}$ for different pair of $e,e'$ and show that it under conditions have sufficient bounds on them. These bounds will allow us to construct curves $\alpha_i$ which must satisfies Lemma \ref{exist-[c]}. Of course to bound $Q$, one must keep in mind that our curve have minimal Dehn twist, otherwise this relative length can be made arbitrarily large or small but won't provide any meaningful bounds of $Q.$\par
 Denote $a_k\in\{g_k,g'_k\}, d_k\in\{c_k,c_k'\}$ for $1\le k\le 2g-2.$ We make the convention that given $a_k$ and $d_k$ we set $a'_k,d_k'$ to be the other arc in the
 pair collection. \par
 Assume that the lemma is false.
There are several cases that we need to consider. 
We first show that there must exists a $P$ of $R_g$ such that there exists $P^k\in P$ with $\frac{\ell(d_k)}{\ell(b_k)}>\log^{4/5}(2g):$
\begin{lem}\label{P^k}
Assume Lemma \ref{exist-[c]} is not true. For every $P$ there must exists some $P^k\in P$ such that,  $\frac{\ell(d_k)}{\ell(b_k)}\ge\log^{4/5}(2g).$
\end{lem}
\begin{proof}
We prove by contradiction. So we assume all $P$ pants decompositions of $R_g$ have a $P_k$ such that  $\frac{\ell(d_k)}{\ell(b_k)}<\log^{4/5}(2g).$\par
Let $\kappa>1.$ Consider the following two cases: 
\begin{itemize}
\item $(\bf A):$ $\ell(a_k)\ell(b_k)\ge \kappa\log(2g)$, for all $k.$
\item
$(\bf B):$ There exists some $k$ such that $\ell(a_k)\ell(b_k)<\kappa\log(2g).$ \end{itemize}
Case $(\bf A):$ Cut $P_k$ into hexagon and it follows from hyperbolic hexagonal \cite{Berdon} formulas we have, $\sinh(\ell(e_k))\sinh(\ell(b_k)/2)=\cosh(\ell(a_k)).$
This gives:
\[\ell(e_k)\le\ell(a_k)+\sinh^{-1}(\frac{3}{\ell(b_k)})\]
which implies, 
\[\frac{\ell(a_k)}{\ell(e_k)}\ge \frac{\ell(a_k)}{\ell(a_k)+\sinh^{-1}(\frac{1}{\ell(b_k)})}.\]
Since $\ell(a_k)\ell(b_k)\ge \kappa\log(2g),$ and $g\ge 3$ we have from the above inequality, 
\[\frac{\ell(a_k)}{\ell(e_k)}\ge \frac{\ell(a_k)}{\ell(a_k)+\sinh^{-1}(\frac{1}{\ell(b_k)})}>f(\kappa).\]
Note $f(\kappa)$ is increasing function of $\kappa$ and $f(\kappa)<1.$\par 
Let $e_{\hat k}$ be such that $\ell(e_{\hat k})=\min_{1\le k\le 2g-2}\{\ell(e_k)\}.$ If $\hat k\le g-1$ then we set $\alpha_1=d_1\cup d_{1'}.$ Here $d_{k'}$ is the cut image curve of $d_k,$ so $\ell(d_k)=\ell(d_{k'}).$ On the other hand, if $\hat k>g-1$ then, we set $\alpha_1=d_{2g-2}\cup d_{2g-2'}.$ In either case we set,
$\alpha_i=2e_{i-1}\cup 2e_{i}$ for $i\ge 2.$\par
It follows from our choice of $\alpha_i$ we have, the curve homotopic to $\beta_i$ must have arcs homotopic to $a_i$ arcs.  
 Now if $\hat k\le g-1$ then set $\bar\beta_{1}=\cup_{2g-2\ge i\ge \hat k}(a_i\cup a_{i'}),$ and if $\hat k> g-1$ then we set
$\bar\beta_{1}=\cup_{1\le i\le \hat k}(a_i\cup a_{i'}).$ And for $i\ge 2$ we also set $\bar\beta_i=\cup_{2g-2\ge i}(a_i\cup a_{i'}).$
Then it follows that we must have,
\[ \frac{\ell(\beta_{1})}{\ell(\alpha_{1})}\ge\frac{\ell(\bar\beta_{1})}{\ell(\alpha_{1})}\ge\sum_{1\le i\le g-1}f(\kappa)=f(\kappa) (g-1). \]
Hence we have,
\[\sum_{1\le i\le 2g-2}\frac{\ell^2(\beta_i)}{\ell^2(\alpha_i)}>\left( f(\kappa)(g-1)\right)^2, \quad\mbox{ for }\quad g\ge 3.\]
\par
Since $\lim_{\kappa\to\infty}f(\kappa)=1$ is increasing function so, there exists $\kappa_o$ such that for $\kappa\ge\kappa_o$ we have for $g=3$,
$(f(\kappa)(2))^2>\frac{3\pi}{2}\log(6).$ Now by the fact that,
\[ \frac{\pi\left( f(\kappa)(g-1)\right)^2}{2g\log(2g)},\quad\mbox{is increasing function of}\quad g.\]
Hence we have $Q([c])>\frac{2}{\pi}g\log(2g)$ for $[c]$ consists of the chosen curves, which is a contradiction. \par
\text{}\\
Next we consider case $(\bf B):$  $\frac{\kappa\log(2g)}{\ell(a_{k_m})}> \ell(b_{k_m})>\frac{\ell(d_{k_m})}{\log^{4/5}(2g)}$ for some $k_m.$ \par 
Now if $\ell(a_{k_m})\ge\kappa\log(2g)$ for all the $k_m$ then,  
\[\frac{\ell(a_{k_m})}{\ell(d_{k_m})}>\frac{\ell(a_{k_m})}{\log^{4/5}(2g)}>1.\]
Let $\hat k_m$ such that $\ell(d_{\hat k_m})=\min_{k_m}\ell(d_{k_m}).$
Here we choose as following:\par
Let $|\{k_m\}|$ denote number of elements of the collection. If $|\{k_m\}|\ge g-1$ then we choose,
\[ \alpha_1=d_{\hat k_m}\cup d_{\hat k_m},\bar\beta_1=\cup_{ j\le 2g-2}a_j\cup a_{j'},\]
and $\alpha_j=2e_j\cup 2e_{j+1}$ for $j\not=k.$ \par
From our choice of $\alpha_i,$ the curve $\beta_1$ must have arcs homotopic to $a_j$ curves. Hence $\ell(\beta_1)\ge\ell(\bar\beta_1)$ and
we have,
\[\frac{\ell(\beta_1)}{\ell(\alpha_1)}\ge\frac{\ell(\bar\beta_1)}{\ell(\alpha_1)}>\sum_{1\le j\le 2g-2}1=2g-2.\]
Hence for this basis we have, $\sum_{1\le i\le g} Q([c]_i)>(2g-2)^2.$ Since 
\[\frac{(2g-2)^2}{g\log(2g)}> 1, \quad g\ge 3,\]
we have contradiction.\par
On the other hand if $|\{k_m\}|< g-1$ then we choose in combination with case $(A)$:
Let $\{d_{\tilde k}\}$ be elements of $\{d_{k_m}\}$ such that among all curves $\beta_{i}$ defined in $(A)$ consist of $a_k\cup a_{k'}$ which do not intersect 
$d_{\tilde k},$ gives $\frac{\ell(\bar\beta_{\hat k})}{\ell(\alpha_{\hat k})}$ maximal value. Let $d_{ k^*}$ be the minimal length curve of $d_{\tilde k}.$ 
Then we have,
\begin{align*} 
\sum_{1\le i\le g}\frac{\ell^2(\beta_i)}{\ell^2(\alpha_i)}&\ge\frac{\ell^2(\bar\beta_{\hat k})}{\ell^2(\alpha_{\hat k})}+\frac{\ell^2(\bar\beta_{ k^*})}{\ell^2(\alpha_{\bar k})}\\
&> \left(f(\kappa)(g-\frac{|\{k_m\}|}{2}-1)\right)^2+\left(\frac{|\{k_m\}|}{2}-1\right)^2
\end{align*} 
Hence by previous estimates we have $Q([c])$ satisfies the inequality, which gives contradiction.\par
Now suppose that $\ell(a_{k_m})<\kappa\log(2g)$ for some of the $m.$ We consider this as the Case $(\bf C).$ Here we breakdown the case $(\bf C)$ into two subcases:
\begin{itemize}
\item
 $(\bf C_1):$ $\frac{\ell(a_{k})}{\ell(b_{k})}\ge 1,$ for all $k$
 \item
 $(\bf C_2):$ $\frac{\ell(a_{k})}{\ell(b_{k})}<1,$ some $k_n.$
 \end{itemize}
Consider $(\bf C_1):$ In this case we have by our global condition $\ell(b_{k})>\frac{\ell(d_{k})}{\log^{4/5}(2g)}$ implies, $\frac{\ell(a_{k})}{\ell(d_{k})}>\frac{1}{\log^{4/5}(2g)}.$ \par
We set $\bar\beta_1=\cup_{2g-2\ge j} a_j\cup a_{j'},$ $\alpha_1=d_{\hat k}\cup d_{\hat k},$ where $\ell(d_{\hat k})=\min_{k}\{\ell(d_k)\}.$ And we set $\alpha_i=2e_{i-1}\cup 2e_{i}$ for $i\ge 2.$ Choose 
$\beta_{i\ge 2}$ so that $\{\alpha_i,\beta_i\}_{1\le i\le g}$ form a basis. Similarly as before we have,
\[\frac{\ell(\beta_{1})}{\ell(\alpha_{1})}\ge\frac{\ell(\bar\beta_{1})}{\ell(\alpha_{1})}>\frac{2g-2}{\log^{4/5}(2g)}.\]
Since,
\[\frac{\pi(2g-2)^2}{2g\log^{14/5}(2g)}>1, \quad\mbox{ for }\quad g\ge 2\]
and it is increasing function of $g,$ 
it follows that we have $Q([c])$ satisfies the inequality which give us a contradiction.\par
\text{}\\
Consider $(\bf C_2)$. In this case we have, $\ell(a_{k_n})<\ell(b_{k_n}).$ \par 
By the hyperbolic identity $\sinh(\ell(e_{k}))\sinh(\frac{\ell(b_{k})}{2})=\cosh(\ell(a_{k}))$ we have,
\[\frac{\ell(a_{k_n})}{\ell(e_{k_n})}>\frac{\ell(a_{k_n})}{\sinh^{-1}\left(\frac{\cosh(\ell(a_{k_n}))}{\sinh(\frac{\ell(a_{k_n})}{2})}\right)}.\]
Here we further subdivide into cases as: 
\begin{itemize}
\item $(\bf C'_2):$ $\ell(a_{k_n})\ge \frac{\rho}{\sqrt[3]{2g}}.$
\item $(\bf C''_2):$ $\ell(a_{k_n})<\frac{\rho}{\sqrt[3]{2g}}.$ 
\end{itemize}
For $(\bf C'_2),$ we have from the above inequality after some simple computations we have, $\frac{\ell(a_{k_n})}{\ell(e_{k_n})}>\left(\frac{\rho}{\sqrt[3]{2g}}\right)^{\frac{11}{10}}.$ Note that this inequality holds for $\ell(a_{k_n})\ge\frac{\rho}{\sqrt[3]{2g}}$ and  $g\ge 3.$ \par
Set $\alpha_1=d_{\hat k}\cup d_{k'}, \bar\beta_1=\cup_{1\le j\le 2g-2}a_j$ and $\alpha_i=e_{2i-1}\cup e_{2i}, \bar\beta_i=\cup_{i\le j\le 2g-2} a_j\cup a_{j'}$ and 
$\bar\beta_{i-1}=\cup_{j\le i-1} a_j\cup a_{j'},$ for $i\not=\hat k.$ 
Then by similar computations as above then show that this basis $[c]$ satisfies,
\[ \sum_{1\le i\le g}\frac{\ell^2(\bar\beta_i)}{\ell^2(\alpha_i)}\ge \frac{4(g-|n|-1)^2}{\log^2(2g)}+\frac{\rho^{\frac{22}{10}}|n|^2}{(2g)^{11/30}} .\]
This inequality follows from that the maximal number of $a_{k_n}$ arcs that don't pass through $d_{\hat k}$ is $|n|/2.$ Here $|n|$ is the number of $k_n.$\par
This implies that, if $|n|>\frac{g}{2}$ then, we can always choose curves $\alpha_i$ as above so that, 
$\sum_{1\le i\le g}\frac{\ell^2(\beta_i)}{\ell^2(\alpha_i)}>\frac{\rho^{\frac{22}{10}}g^{49/30}}{2^{71/30}}.$ By setting $\rho=\frac{6}{10}$ we have,
\[(\pi\frac{\rho^{\frac{22}{10}}g^{49/30}}{2^{71/30}})/(2g\log(2g))>1,\quad\mbox{ for}\quad g\ge 3.\] 
It follows that  $Q([c])> $ gives our contradiction. \par
\text{}\\
Case $(\bf C_2'').$ Here we have, $\ell(a_{k_n})<\frac{\rho}{\sqrt[3]{2g}}.$ Note by the global condition $\ell(d_k)<\ell(b_k)\log^{4/5}(2g)$ we also have, 
$\ell(b_{k_n})<\frac{\kappa\log(2g)}{\ell(a_{k_n})}.$
\par
we have by hexagonal hyperbolic formula,
\[\cosh(\ell(a_{k_n}))=\sinh(\ell(a_{k_n}))\sinh(b_{k_n})\cosh(\ell(d_{k_n}))-\cosh(\ell(a_{k_n}))\cosh(\ell(b_{k_n})) \]
gives
\[\ell(d_{k_n})=\cosh^{-1}\left(\frac{\cosh(\ell(a_{k_n}))(1+\cosh(\ell(b_{k_n}))}{\sinh(\ell(a_{k_n}))\sinh(\ell(b_{k_n}))}\right).\]
by our conditions this implies,
\begin{align*}
\ell(d_{k_n})&\ge \cosh^{-1}\left(\frac{\cosh(\ell(a_{k_n}))(1+\cosh(\frac{\kappa\log(2g)}{\ell(a_{k_n})}))}{\sinh(\ell(a_{k_n}))\sinh(\frac{\log(2g)}{\kappa\ell(a_{k_n})})}\right)\\
&\ge\cosh^{-1}\left(\frac{\cosh(\frac{\frac{6}{10}}{\sqrt[3]{2g}})(1+\cosh(\frac{\kappa\log(2g)}{\ell(a_{k_n})}))}{\sinh(\frac{\frac{6}{10}}{\sqrt[3]{2g}})\sinh(\frac{\kappa\log(2g)}     
{\ell(a_{k_n})})}\right)\\
&> \frac{17}{10}\log(2g)
\end{align*}
By $\ell(d_k)<\ell(b_k)\log^{4/5}(2g)$ implies that, $\ell(b_{k_n})>\frac{17}{10}\log^{1/5}(2g).$ We have,
\[\frac{\ell(b_{k_n})}{\ell(e_{k_n})}>\frac{(2)\frac{17}{10}\log^{1/5}(2g)}{\sinh^{-1}\left(\frac{\cosh(\frac{\frac{6}{10}}{\sqrt[3]{2g}})}
{\sinh(\frac{17}{10}\log^{1/5}(2g))}\right)} \]
\par
With some computations one shows that,  $\frac{\ell(b_{k_n})}{\ell(e_{k_n})}>3\log(2g).$ \par
Here we set $\alpha_1=\cup_{1\le j\le 2g-2}a_j $ and $\alpha_i=2e_{2i-1}\cup 2e_{2i}$ for all the $i$ that within the collection of $P_{k_n}.$ We select the remaining
curves $\alpha_j$ which are not consist of $e_j.$ By construction it follows that, $\beta_i$ the geodesic curve must have arcs homotopic to $a_i$ and,
$\sum_{i\le i\le g}\frac{\ell(\beta_i)}{\ell(\alpha_i)}>\frac{9|n|}{4}\log^2(2g).$ Now if $|n|\ge\frac{g}{2}$ then we have,
$Q([c])>\frac{2}{\pi}g\log(2g).$ On the other hand, if $|n|<\frac{g}{2}$ we can then choose curves given by previous construction. Hence we always curves that gives our contradiction, this completes our prove of the result. 
\end{proof}
Now we establish the induction process of showing:
\begin{cor}\label{P^{k-i}}
Assume Lemma \ref{exist-[c]} is false.
Let $1< i\le g$. Suppose there exists $P^{k_j}$ for every $j< i$ such that,  $\frac{\ell(d_{k_j})}{\ell(b_{k_j})}\ge\log^{4/5}(2g-2j+2).$  Then
there must exists  $P^{k_i}\in P-\cup_{j< i}{P^{k_j}}$ such that $\frac{\ell(d_{k_i})}{\ell(b_{k_i})}\ge\log^{4/5}(2g-2i+2).$
\end{cor}
\begin{proof}
The pants $P^{k_i}\in P-\cup_{ j< i}{P^{k_j}}$ consists of decomposition of surface of genus $g-i+1.$ 
It follows from Lemma \ref{P^k}, there exists $d_{k_i},b_{k_i}$ such that $\frac{\ell(d_{k_i})}{\ell(b_{k_i})}\ge\log^{4/5}(2(g-i+1)).$
\end{proof}
Now we can finish the proof of Lemma \ref{exist-[c]}: If the Lemma is false then, we choose the collection of geodesic curves on surface $R_g$ represented by $\alpha_i=b_i\cup b_{i'}$ for $1\le i\le g.$ Since the shortest geodesic $\beta_i$ intersecting $\alpha_i$ is given by either $c_i\cup c'_i$ or $c'_i\cup c'_{i'},$ i.e.
$\beta_i=d_i\cup d_{i'}.$ Hence it follows from Corollary \ref{P^{k-i}}
we have that, 
\[\frac{\ell(\beta_i)}{\ell(\alpha_i)}\ge\log^{4/5}(2g-2i+2), \quad\mbox{for} \quad 1\le i\le g. \]
After simple computation one shows that, $\sum_{i=1}^{g}\log^{8/5}(2g-2i+2)>\frac{2}{\pi}g\log(2g).$ Hence $Q([c])>\frac{4}{\pi}g\log(g),$ a contradiction. \end{proof}
\begin{prop}\label{invariant}
 There exists a $[c]\in BH_1$ such that, 
 \[Q(\theta.[c])>\frac{2\lambda_g}{\pi}g\log(g),\quad\text{for all}\quad\theta\in stab(\phi).\]
\end{prop}
\begin{rem}
Note that under the action of $\theta$ we could have $i$ permuted,  hence only the summation of the above inequality is preserved. We call any such $[c]$ 
satisfies the inequality, \emph{positive}. Call $[c]$ \emph{invariantly positive} if having property of Lemma \ref{invariant}. 
\end{rem}
\begin{proof}
We need to examine $\theta$ action on elementary curves on pants decomposition. Note that it is enough to show it is true for generators of $stab(\phi).$
We can written elements of $SP(2g,\mathbb{Z})$ as composition of several types of elementary symplectic matrices.  For elements of $stab(\phi),$
which is subgroup generated by elementary matrices which do not intertwines the $\alpha_i$ and $\beta_j$ basis. These elementary matrix correspond to a Nielsen transformation on the generators of $\phi([c]).$\par
The idea is that, take $[c]$ positive which exists by Lemma \ref{exist-[c]}, by similar computations as in the proof of Lemma \ref{exist-[c]} using elementary arcs, if $\theta.[c]$ is not positive then, we can appropriately modify the original curve of $[c]$ to a new $[c']$ so that this, $[c']$ will only increase $Q$ under $\theta$
This is achieved by compare different elementary arc length as we have done previously.\par
Let $E_{ij}$ denote a elementary matrix such that $E_{ij}$ map $[c]$ into basis with $\alpha_i$ replaced by $\alpha_i+\alpha_j$ and $\beta_j$ replaced by
$\beta_j-\beta_i$, and rest unchanged. Example $E_{12}$ for $g=3$ of $SP(6,\mathbb{Z}):$
\[E_{12}=\left(\begin{array}{cccccc} 1&1&0&0&0&0\\0&1&0&0&0&0\\0&0&1&0&0&0\\0&0&0&1&0&0\\0&0&0&-1&1&0\\0&0&0&0&0&1\end{array}\right).\]

We need to consider cases on pair pants. Assume the statement is false and there exists $l,m$ such that $\sum_jQ[E_{lm}.[c]_j]<\frac{2}{\pi}g\log(2g)$ for all
positive $[c].$ It follows from the proof of Lemma \ref{exist-[c]}, it is suffice to consider $\alpha_l,\alpha_m$ are curves consist of elementary arcs. Let $h_{lm}$ denote the geodesic curve representation of homological class $[\alpha_l+\alpha_m].$ Let $[f_{lm}]$ be a class of curves which is canonical dual to $[h_{lm}]$ and non-intersecting to rest of basis. We denote the $f_{lm}$ to be the geodesic representative of $[f_{lm}].$ \par

Take a $[c]$ provided by Lemma \ref{exist-[c]}. We will show that either this $[c]$ is invariantly positive or it implies another $[\hat c]$ is invariantly positive. To do so we will examine
all possible cases of $[c].$ We continue to use our notations and definitions given in the proof of Lemma \ref{exist-[c]}. We need to examine all possible
curves given by the $E_{lm}.[c]$. \par
As in the Lemma \ref{exist-[c]}, first assume that we have $\frac{\ell(d_k)}{\ell(b_k)}<\log^{4/5}(2g)$ for all $P^k.$ Note that we have some positive $[c].$ Since $h$ is geodesic closed curve homotopically connects $b_l,b_m,$ the
length is bounded above by sum of these curves and twice the connecting geodesic curve. As in the proof of Lemma \ref{exist-[c]}, if $\ell(a_k)\ell(b_k)\ge\log(2g)$ for
all $k,$ then $\ell(e_l)+\ell(e_m)>\ell(d_l)+\ell(d_m).$ So if we set $\delta=\ell(h_{lm})-(\ell(b_l)+\ell(b_m)),$ then 
$\ell(f_lm)\ge\ell(e_l)+\ell(e_m)+\delta.$ Hence we have, $\frac{\ell(f_{lm})}{\ell(h_{lm})}\ge\frac{g}{\log(2g)},$ and we have $E_{lm}.[c]$ positive.  
On the other hand if, $\ell(a_k)\ell(b_k)<\log(2g)$ for some $P^k$ contains $b_l,b_m$ then, we have a curve $\alpha$ which consists of $a_k$ as elementary arcs and $\frac{\ell(d_j)}{\ell(\alpha)}>\frac{g}{\log(2g)},j\in\{l,m\}.$  We replace one of the original $\alpha_l$ with $\alpha$ to form basis $[\hat c].$ It is obvious that $[\hat c]$ is positive. \par
For simplicity we use same notation for the geodesic representation curves for $[\hat c].$ Then $h$ have $\alpha_l$ a Dehn twist around $b_m$ which trace off from 
$a_k$ arcs. The geodesic $f_{lm}$ consists of arcs homotopic to arcs of $\beta_m,\beta_l,b_l.$ Hence $\ell(f_{lm})\ge \frac{1}{2}(\ell(\beta_m)+\ell(\beta_l))+\ell(b_l).$
This implies we have lower bound of \[\frac{\ell(f_{lm})}{\ell(b_l)}\ge\frac{\ell(\beta_m)}{\ell(b_l)}+1.\]
Hence we have, $E_{lm}.[\hat c]$ positive. By repeat this replacement for rest of curves of $[c]$ if necessary we can claim that $E_{lm}.[\hat c];1\le l,m\le g.$ satisfies same inequality. Note that $E_{lm}.E_{l'm'}.[\hat c]$ will increase Dehn twist which will increase the lower bound by twist number. i.e. denote $f_{l_1m_1...l_jm_j}$
be the geodesic representative of $[f]$ corresponds to $E_{l_im_1}....E_{l_jm_j}.[\hat c]$ we have,
 \[\frac{\ell(f_{l_1m_1...l_jm_j})}{\ell(b_l)}\ge\frac{\ell(\beta_m)}{\ell(b_l)}+j.\]
Hence we have, $[\hat c]$ is invariantly positive. \par
Now it follows from Corollary \ref{P^{k-i}} and Lemma \ref{P^k} that, every pants decompositions have cut curves $\alpha_i$ such that, $\beta_i$ have the property 
of $\frac{\ell(\beta_i)}{\ell(\alpha_i)}\ge\log^{4/5}(2g), 1\le i\le g.$ This implies that there exists a invariantly positive $[c].$

\end{proof}
From Lemma \ref{exist-[c]} and Lemma \ref{invariant} we can prove the following inequality of the Schottky length,
\begin{prop}\label{add-length}
There exists $[c]\in BH_1$ such that, 
\[\mathscr{T}_{[c]}> \lambda_g\log(g).\]
\end{prop}
\begin{proof}
Let $[c]$ be given by Lemma \ref{exist-[c]}. By Lemma \ref{mark-length} we have,
\[T_{[c]_i}\ge\frac{\pi}{2}Q([c]_i).\]
Since $\sum_iQ([c]_i)>\frac{2\lambda_g}{\pi}g\log(g),$ this implies that $\sum^{g}_{i=1}T_{[c]_i}>\lambda_g g\log(g).$
Since $\mathscr{T}_{[c]}=\frac{1}{g}\sum^{g}_{i=1}T_{[c]_i},$
hence the above strict inequality implies the result.
\end{proof}
\section{Proof of Theorem \ref{main1}}
Let $[\![c]\!]\in\hat BH_1.$ For $[c]\in [\![c]\!],$ $\omega_{[c]}$ is an generators set of $\G_{[c]}.$ Since $\G_{[c]}=\phi([c]), \forall [c]\in [\![c]\!]$ and
$\bigcup_{[c]\in[\![c]\!]}\omega_{[c]}=\mathbb{W}_\G$ we have our next corollary,
\begin{prop}\label{T-length}
Let $[c]$ be given by Proposition \ref{add-length}. Then $\|\mathbb{W}_\G\|_x>\lambda_g\log(g).$
\end{prop}
\begin{proof}
It follows from Proposition \ref{add-length} and $\mathfrak{w}_x(\omega_{[c]})=\mathscr{T}_{[c]},$ we have:
\[\mathfrak{w}_x(\omega_{[c]})>\lambda_g\log(g).\]
Since $\bigcup_{[c]\in[\![c]\!]}\omega_{[c]}=\mathbb{W}_\G,$ by the invariant positivity result  of Lemma \ref{invariant} we have :
\[\mathfrak{w}_x(\omega_{[c]})>\lambda_g\log(g),\quad \forall [c]\in [\![c]\!].\]
Hence we have:
\[\|\mathbb{W}_\G\|_x=\inf_{\omega_{[c]}\in\mathbb{W}_\G}\mathfrak{w}_x(\omega_{[c]})>\lambda_g\log(g).\]
\end{proof}

\begin{proof}{Theorem \ref{main1}}:\par
For a $R_g,$ it follows from Proposition \ref{T-length}, there exists homological marking $[c]\in BH_1$ of $R_g$ such that, $\G_{[c]}=\phi([c])$, the covering Schottky group satisfies, $\|\mathbb{W}_\G\|_x>\lambda_g\log(g)$ for all $\mathbb{H}^3.$ Hence by Corollary \ref{D-main} and Corollary \ref{sum-main}, and note that for such $[c]$ we have $\lambda<2$, which implies $\mathfrak{D}_{\G_{[c]}}<1.$

\end{proof}
Finally, we give two obvious applications of our theorem. The first application Corollary \ref{class-H}, address a folklore question that was originally due to Bers.
\begin{cor}\label{class-H}
If $H_c\ge 1$ then, for all $[R_g]\in\mathscr{M}_g,$ $\pi_S^{-1}([R_g]),$ have classical fiber.
\end{cor}
\begin{proof}
Assume that $H_c\ge 1,$ so $\tau_c\le 1.$
Since $\tau_c\le 1,$ Proposition \ref{sys-gen} and Theorem \ref{main2} implies that $\G$ is classical Schottky group.
Hence for $[R_g]\in\mathscr{M}_g,$ $\pi_S^{-1}([R_g])\cap\mathfrak{J}_{g,o}\not=\emptyset.$ Therefore,
$\pi_S|_{\mathfrak{J}_{g,o}}$ is surjective.\par

\end{proof}
The second application is presentation of period matrix of $R_g$ in Schottky coordinates. It's well known theorem of Torelli that there exists a injective map from
$\mathscr{M}_g$ into Siegel's space of symmetric $g\times g$ matrix over $\mathbb{C}.$ This is given by the $P_{mn}$ period matrix of $R_g.$ Many beautiful theorems has been proved on $P_{mn}$, such as Buser-Sarnak's theorem \cite{BS} states that the locus of Jacobians lie in very small neighborhood of the boundary of space of principally polarized abelian varieties for large genus $R_g.$  
\par
It's well known fact \cite{Baker}, that $P_{mn}$ can be represented in local coordinates of $\mathfrak{J}_g.$
Recall that local coordinates of $\mathfrak{J}_g$, which is $3g-3$-dimensional complex manifold, are given by variables $\lambda_i,z_{-,i},z_{+,i}, 1\le i\le g-1$ multiplier
and two fixed points respectively. Recall, given $z_1,z_2,z_3,z_4\in\mathbb{C},$ the cross ratio is: $[z_1,z_2,z_3,z_4]=\frac{(z_1-z_3)(z_2-z_4)}{(z_1-z_4)(z_2-z_3)}.$
Let $\G\in\mathfrak{J}_g$ be generated by $<\g_1,...,\g_g>,$ we denote by $\G_l$ the subgroup of $\G$ generated by $\g_l.$ Then the following \emph{formal} presentation of  $P_{mn}$ is well known and go back to Schottky himself \cite{Baker, MD}:
\[P_{mn}=\sum_{\gamma\in\G_n\setminus\G/\G_m }\log[z_{-,n},\gamma z_{-,m}, z_{+,n},\gamma z_{+,m}]+\delta_{mn}\log\lambda_n. \]
However, even though $P_{mn}$ has been formally known for a very long time but, in general as $P_{mn}$ is infinite sum, so it is not always convergent for a arbitrary  $\G\in\mathfrak{J}_g,$ and it is not known in general. In this respect, we have our second simple application of Theorem \ref{main1}:
\begin{cor}\label{cor1}
Let $[R_g]\in\mathscr{M}_g.$ There exists $\G\in\mathfrak{J}_g$ such that the period matrix of $P_{mn}$ of $[R_g]$ is given by the above 
presentation.
\end{cor}
\begin{proof}
The only thing needs to be verified is that $P_{mn}$ is convergent for some $\G\in\pi^{-1}_S([R_g]).$ By Theorem \ref{main2}, we have a $\hat\G\in\pi^{-1}_S([R_g])$
such that $\mathfrak{D}_{\hat\G}<1.$ \par
Note that $P_{mn}$ is obtained as integral around canonical basis cycles of $H_1(R_g,\mathbb{Z})$ of 
$\hat\G$-invariant holomorphic cocycles on $\Omega_{\hat\G}$: 
\[\omega_n(z)=\sum_{\g\in\hat\G_n\setminus\hat\G}d\log\left(\frac{z-\g z_{+,n}}{z-\g z_{-,n}}\right).\] 
It follows from the residue formula 
we have $P_{mn}.$ \par
By change of variable we have $\omega_n$ is convergent if the Poincare series: $\sum_{\g\in\hat\G_n\setminus\hat\G}|\g'(z)|$ is convergent. Since $\mathfrak{D}_{\hat\G}<1,$
we have  $\sum_{\g\in\hat\G_n\setminus\hat\G}|\g'(z)|$ is convergent. Hence $P_{mn}$ exists for $\hat\G.$
\end{proof}
\text{}\\
E-mail: yonghou@princeton.edu
\pdfbookmark[1]{Reference}{Reference}
\bibliographystyle{plain}

\end{document}